\documentclass[12pt]{amsart}

\usepackage{amsthm}

\usepackage{amssymb}
\usepackage{verbatim}
\usepackage[curve,matrix,arrow]{xy}
\usepackage{amsmath,amsfonts}
\usepackage{graphicx}
\usepackage{epsf}

\newcommand{\thmref}[1]{Theorem~\ref{#1}}
\newcommand{\propref}[1]{Proposition~\ref{#1}}
\newcommand{\lemref}[1]{Lemma~\ref{#1}}
\newcommand{\eqnref}[1]{Equation~(\ref{#1})}
\newcommand{\remref}[1]{Remark~\ref{#1}}
\newcommand{\corref}[1]{Corollary~\ref{#1}}
\newcommand{\exref}[1]{Example~\ref{#1}}
\newcommand{\figref}[1]{Figure~\ref{#1}}

\newcommand{\conjref}[1]{Conjecture~\ref{#1}}

  {\end{list}}

\def\vc{\vec{v}}
\def\Zh{{\mathop{\rm Zh}}}

\def\li{L_{i}}
\def\ri{R_{i}}
\def\ddx{{\frac{d}{dx}}}
\def\oga{{\overline{\ga}}}
\def\tga{{\widetilde{\ga}}}
\def\jpxq{{j^{\beta_i}_{p}(x,q)}}
\def\jqxp{{j^{\beta_i}_{q}(x,p)}}
\def\jxpq{{j^{\beta_i}_{x}(p,q)}}

\def\RR{{\mathbb R}}

\newtheorem{theorem}{Theorem}[section]

\newtheorem{corollary}[theorem]{Corollary}
\newtheorem{conjecture}[theorem]{Conjecture}
\newtheorem{lemma}[theorem]{Lemma}
\newtheorem{proposition}[theorem]{Proposition}

\theoremstyle{example}

\newtheorem{remark}[theorem]{Remark}

\theoremstyle{definition}

\theoremstyle{notation}
\newtheorem*{notation}{Notation}
\newtheorem{example}[theorem]{Example}

\newcommand{\dd}[1]{\delta_{#1}}
\newcommand{\DD}[1]{\Delta_{#1}}
\newcommand{\nn}[1]{\mu_{#1}}
\newcommand{\jj}[3]{j_{#1}(#2,#3)}
\newcommand{\hj}[3]{\hat{j}_{#1}(#2,#3)}

\newcommand{\ga}{\Gamma}
\newcommand{\gam}[1]{\Gamma_{#1}}
\newcommand{\tg}{\tau(\Gamma)}
\newcommand{\ta}[1]{\tau(#1)}

\newcommand{\ee}[1]{E(#1)}
\newcommand{\vv}[1]{V(#1)}
\newcommand{\va}{\upsilon}
\newcommand{\vb}{\text{v} \hspace{0.5 mm}}

\newcommand{\pp}{p_{i}}

\newcommand{\qq}{q_{i}}

\def\can{{\mathop{\rm can}}}

\def\Zh{{\mathop{\rm Zh}}}
\def\CC{{\mathbb C}}

\def\cC{{\mathcal C}}

\def\<{\langle }
\def\>{\rangle }
\newcommand{\secref}[1]{\S\ref{#1}}

\def\elg{\ell (\ga)}

\begin{document}

\title[The tau constant of a metrized graph and graph operations]
{The tau constant of a metrized graph and its behavior under graph operations}

\author{Zubeyir Cinkir}
\address{Zubeyir Cinkir\\
Department of Mathematics\\
University of Georgia\\
Athens, Georgia 30602\\
USA}
\email{cinkir@math.uga.edu}


\thanks{I would like to thank Dr. Robert Rumely for his guidance. His continued support and encouragement made this work possible. I also would like to thank Dr. Matthew Baker for always being available for useful discussions during and before the preparation of this paper. Their suggestions and work were inspiring to me.}

\begin{abstract}
This paper concerns the tau constant, which is an important invariant of
a metrized graph, and which has applications to arithmetic properties of curves. We give several formulas for the tau constant, and show how it changes under graph operations including deletion of an edge, contraction of an edge, and union of graphs along one or two points.
We show how the tau constant changes when
edges of a graph are replaced by arbitrary graphs.
We prove Baker and Rumely's lower bound conjecture on the tau constant for several classes of metrized graphs.
\end{abstract}

\maketitle

\section{Introduction}\label{section introduction}
\vskip .1 in

Metrized graphs, which are graphs equipped with a distance function
on their edges, appear in many places in arithmetic geometry.
R. Rumely \cite{RumelyBook} used metrized graphs
to develop arithmetic capacity theory,
contributing to local intersection theory for
curves over non-archimedean fields. T.
Chinburg and Rumely \cite{CR} used metrized graphs to define their ``capacity pairing".
Another pairing satisfying ``desirable" properties is Zhang's
``admissible pairing on curves", introduced by S. Zhang \cite{Zh1}.
Arakelov introduced an intersection pairing at infinity and used
analysis on Riemann surfaces to derive global results. In the
non-archimedean case, metrized graphs appear as the analogue of a
Riemann surface. Metrized graphs and their invariants are studied in the articles
\cite{Zh1}, \cite{Zh2}, \cite{Fa}, \cite{C1}, \cite{C2}.

Metrized graphs which arise as dual graphs of curves, and Arakelov
Green's functions $g_{\mu}(x,y)$ on the metrized graphs, play an
important role in both of the articles \cite{CR} and \cite{Zh1}. Chinburg and Rumely
worked with a canonical measure $\mu_{can}$ of total mass $1$ on a
metrized graph $\ga$ which is the dual graph of the special fiber of a curve $C$.
Similarly, Zhang \cite{Zh1} worked with an ``admissible measure"
$\mu_{ad}$, a generalization of $\mu_{can}$, of total mass $1$
on $\ga$. The diagonal values $g_{\mu_{can}}(x,x)$ are constant on
$\ga$. M. Baker and Rumely called this constant  the ``tau constant"
of a metrized graph $\ga$, and denoted it by $\tg$. They posed a conjecture
(see \conjref{TauBound}) concerning lower bound of $\tg$. We call it Baker and Rumely's
lower bound conjecture.

In summer 2003 at UGA,  an REU group lead by Baker and Rumely
studied properties of the tau constant and the lower bound conjecture. Baker
and Rumely \cite{BRh} introduced a measure valued Laplacian operator $\Delta$ which
extends Laplacian operators studied earlier in the articles \cite{CR} and
\cite{Zh1}. This Laplacian operator combines the ``discrete''
Laplacian on a finite graph and the ``continuous'' Laplacian
$-f''(x)dx$ on $\RR$. Later, Baker and Rumely \cite{BRh} studied
harmonic analysis on metrized graphs. In terms of spectral theory,
the tau constant is the trace of the inverse operator of $\Delta$, acting on functions $f$ for which
$\int_{\ga}f d\mu_{can}=0$, when $\ga$ has total length $1$.

In this paper, we express the canonical measure $\mu_{can}$ on a metrized graph $\ga$
in terms of the voltage function $j_x(y,z)$ on $\ga$.
Our main focus is to give a systematic study
of how the tau constant behaves under common graph operations. We give new formulas for the tau constant,
and show how it changes under graph operations such as the deletion of an edge, the contraction of an edge into its end points,
identifying any two vertices, and extending or shortening one of the edge lengths of $\ga$. We define a new
graph operation which we call ``full immersion of a collection of given graphs into another graph''
(see \secref{section general DA}), and we show how the tau constant changes under this operation. We prove the lower bound conjecture for several classes of metrized graphs. We show how our formulas can be applied to compute the tau constant for various classes of metrized graphs, including those with vertex connectivity $1$ or $2$. The results here extend those obtained in \cite[Sections 2.4, 3.1, 3.2, 3.3, 3.4 and 3.5]{C1}.
Further applications of these results can be found in the articles \cite{C2}, \cite{C3}, \cite{C4}, and \cite{C5}.
\section{The tau constant and the lower bound conjecture}
\label{section two}

In this section, we first recall a few facts about metrized graphs,
the canonical measure $\mu_{can}$ on a metrized graph $\ga$, the
Laplacian operator $\Delta$ on $\ga$, and the tau constant $\tg$ of
$\ga$. Then we give a new expression for $\mu_{can}$ in terms of the
voltage function and two arbitrary points $p$, $q$ in $\ga$. This enables
us to obtain a new formula for the tau constant. We also show how
the Laplacian operator $\Delta$ acts on the product of two
functions.

A \textit{metrized graph} $\ga$ is a finite connected graph
equipped with a distinguished parametrization of each of its edges.
One can find other definitions of metrized graphs in the articles \cite{BRh}, \cite{Zh1}, \cite{BF}, and the references contained in those articles.

A metrized graph can have multiple edges and self-loops. For any given $p \in \ga$,
the number of directions
emanating from $p$ will be called the \textit{valence} of $p$, and will be denoted by
$\va(p)$. By definition, there can be only finitely many $p \in \ga$ with $\va(p)\not=2$.

For a metrized graph $\ga$, we will denote its set of vertices by $\vv{\ga}$.
We require that $\vv{\ga}$ be finite and non-empty and that $p \in \vv{\ga}$ for each $p \in \ga$ if
$\va(p)\not=2$. For a given metrized graph $\ga$, it is possible to enlarge the
vertex set $\vv{\ga}$ by considering more additional points of valence $2$ as vertices.

For a given graph $\ga$ with vertex set $\vv{\ga}$, the set of edges of $\ga$ is the set of closed line segments with end points in $\vv{\ga}$. We will denote the set of edges of $\ga$ by $\ee{\ga}$.

Let $v:=\# (\vv{\ga})$ and $e:=\# (\ee{\ga})$. We define the \textit{genus} of $\ga$ to be the first Betti number $g:=e-v+1$ of the graph $\ga$. Note that the genus is a topological invariant of $\ga$. In particular, it is independent of the choice of the vertex set $\vv{\ga}$.
Since $\ga$ is connected, $g (\ga)$ coincides with the cyclotomic number of $\ga$ in combinatorial graph theory.

We denote the length of an edge $e_i \in \ee{\ga}$ by $\li$. The total length of $\ga$, which will be denoted by $\elg$, is given by $\displaystyle \elg=\sum_{i=1}^e\li$.

Let $\ga$ be a metrized graph. If we scale each edge of $\ga$ by multiplying its length by $\frac{1}{\ell(\ga)}$, we
obtain a new graph which is called normalization of $\ga$, and will be
denoted $\ga^{N}$. Thus, $\ell(\ga^{N})=1$.

We will denote the graph obtained from $\ga$ by deletion of the interior points of an edge $e_i \in \ee{\ga}$ by $\ga-e_i$. An edge $e_{i}$ of a connected graph $\ga$ is called a \textit{bridge} if
$\ga-e_{i}$ becomes disconnected. If there is no such edge in $\ga$,
it will be called a \textit{bridgeless graph}.

As in the article \cite{BRh},
$\Zh(\Gamma)$ will be used to denote the set of all
continuous functions $f : \Gamma \rightarrow \CC$ such that for some vertex set $\vv{\ga}$, $f$ is
$\cC^2$ on $\ga
\backslash \vv{\ga}$ and $f^{\prime \prime}(x) \in L^1(\Gamma)$.

Baker and Rumely \cite{BRh} defined the following measure valued \textit{Laplacian} on a given metrized graph.
For a function $f \in \Zh(\Gamma)$,
\begin{equation}
\DD{x}(f(x))=-f''(x)dx - \sum_{p \in \vv{\ga}}\bigg[ \sum_{\vc
\hspace{0.5 mm} \text{at} \hspace{0.5 mm} p}
d_{\vc}f(p)\bigg]\dd{p}(x),
\end{equation}
See the article \cite{BRh} for details and for a description of the largest class of functions for which a measure valued Laplacian can be defined.

We will now clarify how the Laplacian operator acts on a product of
functions. For any two functions $f(x)$ and $g(x)$ in $\Zh(\Gamma)$,
we have $f(x) g(x) \in \Zh(\Gamma)$ and
\begin{equation*}
\begin{split}
\DD{x}(f(x)g(x)) &= -\big[f''(x)g(x)+2f'(x)g'(x)+f(x)g''(x)\big]dx
\\ & \hspace{5 mm}
- \sum_{p \in \vv{\ga}}\bigg[ \sum_{\vc  \hspace{0.5 mm} \text{at}
\hspace{0.5 mm} p} (f(p)d_{\vc}g(p)+g(p)d_{\vc}f(p)\bigg]\dd{p}(x)
\\ &= -g(x)f''(x)dx - \sum_{p \in \vv{\ga}}g(p)\bigg[ \sum_{\vc  \hspace{0.5 mm}
\text{at} \hspace{0.5 mm} p} d_{\vc}f(p)\bigg]\dd{p}(x)
\\ & \hspace{5 mm}
-f(x)g''(x)dx-\sum_{p \in \vv{\ga}}f(p)\bigg[ \sum_{\vc  \hspace{0.5
mm} \text{at} \hspace{0.5 mm} p} d_{\vc}g(p)\bigg]\dd{p}(x) -2f'(x)g'(x)dx
\\ &=g(x)\DD{x}f(x)+f(x)\DD{x}g(x)-2f'(x)g'(x)dx.
\end{split}
\end{equation*}
Thus, we have shown the following result:
\begin{theorem}\label{thmdelta}
For any $f(x)$ and $g(x)$ $\in \Zh(\Gamma)$, we have
\begin{equation*}
\DD{x}(f(x)g(x)) = g(x)\DD{x}f(x)+f(x)\DD{x}g(x)-2f'(x)g'(x)dx .
\end{equation*}
\end{theorem}
The following proposition shows that the Laplacian on $\Zh(\Gamma)$
is ``self-adjoint'', and explains the choice of sign in the
definition of $\Delta$. It is proved by a simple integration by
parts argument.
\begin{proposition}\cite[Lemma 4.a]{Zh1}\cite[Proposition 1.1]{BRh}\label{prop Greens identity}
\label{SelfAdjointProp} For every $f,g \in \Zh(\Gamma)$,
\begin{equation*}
\begin{split}
\int_\Gamma \overline{g} \, \Delta f &= \int_\Gamma f \,
\overline{\Delta g}, \quad \text{ Self-Adjointness of $\Delta$}
\\ &= \int_\Gamma f'(x) \overline{g'(x)}
dx
\quad \text{ Green's Identity}.
\end{split}
\end{equation*}
\end{proposition}
In the article \cite{CR}, a kernel $j_{z}(x,y)$ giving a
fundamental solution of the Laplacian is defined and studied as a
function of $x, y, z \in \Gamma$. For fixed $z$ and $y$ it has the
following physical interpretation: when $\Gamma$ is viewed as a
resistive electric circuit with terminals at $z$ and $y$, with the
resistance in each edge given by its length, then $j_{z}(x,y)$ is
the voltage difference between $x$ and $z$, when unit current enters
at $y$ and exits at $z$ (with reference voltage 0 at $z$).

For any $x$, $y$, $z$ in $\ga$, the voltage function $j_x(y,z)$ on
$\ga$ is a symmetric function in $y$ and $z$, and it satisfies
$j_x(x,z)=0$ and $j_x(y,y)=r(x,y)$, where $r(x,y)$ is the resistance
function on $\ga$. For each vertex set $\vv{\ga}$, $j_{z}(x,y)$ is
continuous on $\ga$ as a function of $3$ variables.
As the physical interpretation suggests, $j_x(y,z) \geq 0$ for all $x$, $y$, $z$ in $\ga$.
For proofs of these facts, see the articles \cite{CR}, \cite[sec 1.5 and sec 6]{BRh}, and \cite[Appendix]{Zh1}.
The voltage function $j_{z}(x,y)$ and the resistance function $r(x,y)$ on a metrized graph
were also studied by Baker and Faber \cite{BF}.
\begin{proposition}\cite{CR} \label{prop cordjpq}
For any $p,q,x \in \ga$,
$ \quad \DD{x}\jj{p}{x}{q}=\dd{q}(x)-\dd{p}(x)$.
\end{proposition}
In \cite[Section 2]{CR}, it was shown that the theory of harmonic functions on metrized graphs is equivalent to the theory of resistive electric circuits with terminals. We now recall the following well known facts from circuit theory. They will be used frequently and implicitly in this paper and in the papers \cite{C2}, \cite{C3}, \cite{C4}. The basic principle of circuit analysis is that if one subcircuit of a circuit is replaced by another circuit which has the same resistances between each pair of terminals as the original subcircuit,
then all the resistances between the terminals of the original circuit are unchanged. The following subcircuit replacements are particularly useful:

\textbf{Series Reduction:} Let $\ga$ be a graph with vertex set $\{p, q, s \}$. Suppose that $p$ and $s$ are connected
 by an edge of length $A$, and that $s$ and $q$ are connected by an edge of length $B$. Let $\beta$ be a graph with vertex set $\{p, q\}$, where $p$ and $q$ are connected by an edge of length $A+B$.
 Then the effective resistance in $\ga$ between $p$ and $q$ is equal to the effective resistance in $\beta$ between $p$ and $q$.
These are illustrated by the first two graphs in \figref{fig seriesandparallel}.
\begin{figure}
\centering
\includegraphics[scale=0.8]{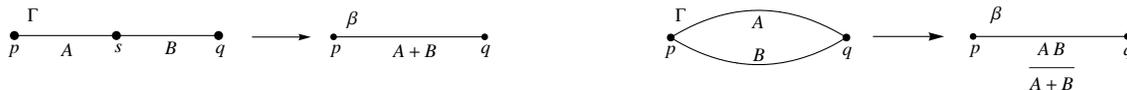} \caption{Series and Parallel Reductions} \label{fig seriesandparallel}
\end{figure}

\textbf{Parallel Reduction:} Suppose $\ga$ and $\beta$ be two
graphs with vertex set $\{p, q\}$. Suppose $p$ and $q$ in $\ga$ are connected by two edges of lengths $A$ and $B$, respectively, and
let $p$ and $q$ in $\beta$ be connected by an edge of length $\frac{A B}{A+B}$ (see the last two graphs in \figref{fig seriesandparallel}).
Then the effective resistance in $\ga$ between $p$ and $q$ is equal to the effective resistance in $\beta$ between $p$ and $q$.

\textbf{Delta-Wye transformation:} This is the one case where a mesh
can be replaced by a star. Let $\ga$ be a triangular graph
with vertices $p$, $q$, and $s$. Then, $\ga$ (with resistance function $r_{\ga}$)
can be transformed to a Y-shaped graph $\beta$ (with resistance
function $r_{\beta}$) so that $p, q, s$ become end points in $\beta$ and the
following equivalence of resistances hold:
$r_{\ga}(p,q)=r_{\beta}(p,q)$, $r_{\ga}(p,s)=r_{\beta}(p,s)$, $r_{\ga}(q,s)=r_{\beta}(q,s).$
Moreover, for the resistances $a$, $b$, $c$  in $\ga$, we have the resistances $\frac{bc}{a+b+c}$, $\frac{ac}{a+b+c}$, $\frac{ab}{a+b+c}$ in $\beta$, as illustrated by the first two graphs in \figref{fig deltawyewyedelta2}.

\textbf{Wye-Delta transformation:} This is the inverse
Delta-Wye transformation, and is illustrated by the last two graphs in  \figref{fig deltawyewyedelta2}.

\begin{figure}
\centering
\includegraphics[scale=0.8]{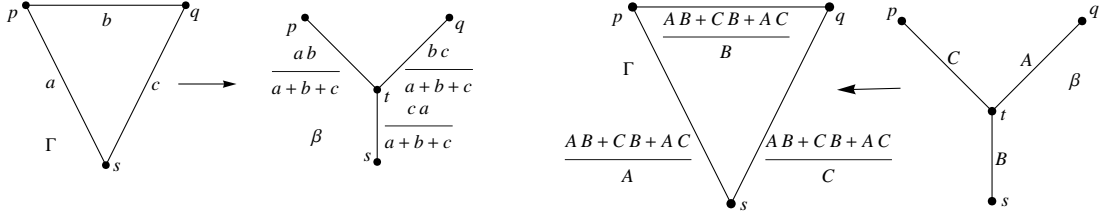} \caption{Delta-Wye and Wye-Delta transformations} \label{fig deltawyewyedelta2}
\end{figure}

\textbf{Star-Mesh transformation:} An $n$-star shaped graph (
i.e. $n$ edges with one common point whose other end points are of
valence $1$)
can be transformed into a complete graph of $n$ vertices (which does
not contain the common end point) so that all resistances between the remaining
vertices remain unchanged. A more precise description is as follows:

Let $L_1, L_2, \cdots, L_n$ be the edges in an $n$-star shaped graph $\ga$ with
common vertex $p$, where $L_i$ is the length of
the edge connecting the vertices $q_i$ and $p$ (i.e., the resistance between the
vertices $q_i$ and $p$. The star-mesh transformation applied to $\ga$ gives
a complete graph $\ga^{c}$ on the set of vertices $q_1,q_2, \cdots, q_n$ with $\frac{n(n-1)}{2}$ edges.
Let $L_{ij}$ be the length of the edge connecting the vertices $q_i$ and $q_j$ in $\ga^{c}$ for any
$1\leq i < j \leq n$. Then
$\displaystyle L_{ij}=L_i L_j \cdot \sum_{k=1}^n \frac{1}{L_k}.$
When $n=2$, the star-mesh transformation is identical to series reduction.
When $n=3$, the star-mesh transformation is identical to the Wye-Delta transformation,
and can be inverted by the Delta-Wye transformation.
When $n\geq 4$, there is no inverse transformation for the star-mesh
transformation. \figref{fig starmesh} illustrates the case $n=6$.
(For more details see \cite{S} or \cite{F-C}).
\begin{figure}
\centering
\includegraphics[scale=0.8]{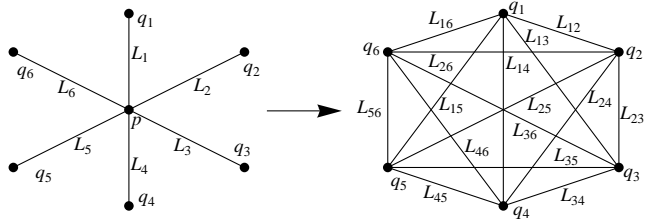} \caption{Star-Mesh transformations when $n=6$.} \label{fig starmesh}
\end{figure}

For any given $p$ and $q$ in $\ga$, we say that an edge $e_i$ is not part of a simple path from $p$ to $q$ if all walks
starting at $p$, passing through $e_i$, and ending at $q$ must visit some vertex more than once.
Another basic principle of circuit reduction is the following transformation:

The effective resistances between $p$ and $q$ in both $\ga$ and $\ga-e_i$ are the same if $e_i$ is not part of a simple
path from $p$ to $q$. Therefore, such an edge $e_i$ can be deleted as far as the resistance between $p$ and $q$ is concerned.

For any real-valued, signed Borel measure $\mu$ on $\Gamma$ with
$\mu(\Gamma)=1$ and $|\mu|(\Gamma) < \infty$, define the function
$\displaystyle j_{\mu}(x,y) \ = \ \int_{\Gamma} j_{\zeta}(x,y) \, d\mu({\zeta}).$
Clearly $j_{\mu}(x,y)$ is symmetric, and is jointly continuous in
$x$ and $y$. Chinburg and Rumely discovered in \cite{CR} that there is a unique real-valued, signed Borel measure $\mu=\mu_{can}$ such that $j_{\mu}(x,x)$ is constant on $\ga$. The measure $\mu_\can$ is called the
\textit{canonical measure}.
Baker and Rumely \cite{BRh} called the constant $\frac{1}{2}j_{\mu}(x,x)$ the \textit{tau constant} of $\ga$ and denoted by $\tg$. In terms of spectral theory, as shown in the article \cite{BRh}, the tau constant $\tg$ is the trace of the inverse of the Laplacian operator on $\ga$ with respect to $\mu_{can}$.

The following lemma gives another description of the tau constant. In particular, it implies that the tau constant is positive.
\begin{lemma}\cite[Lemma 14.4]{BRh}\label{lemtauformula}
For any fixed $y$ in $\ga$,
$\displaystyle \tg =\frac{1}{4}\int_{\ga}\big(\frac{\partial}{\partial x} r(x,y) \big)^2dx$.
\end{lemma}
The canonical measure is given by the following explicit formula:
\begin{theorem}\cite[Theorem 2.11]{CR} \label{thmCanonicalMeasureFormula}
Let $\ga$ be a metrized graph. Suppose that $\li$ is the length of edge $e_i$ and $R_i$ is the effective
resistance between the endpoints of $e_i$ in the graph $\Gamma-e_i$,
when the graph is regarded as an electric circuit with resistances
equal to the edge lengths. Then we have
\begin{equation*}
\mu_\can(x) \ = \ \sum_{p \in \vv{\ga}} (1 - \frac{1}{2}\vb(p))
\, \delta_p(x) + \sum_{e_i \in \ee{\ga}} \frac{dx}{L_i+R_i},
\end{equation*}
where
$\delta_p(x)$ is the Dirac measure.
\end{theorem}
\begin{corollary}\cite[Corollary 14.2]{BRh}
The measure $\mu_\can$ is the unique measure $\nu$ of total mass 1
on $\Gamma$ maximizing the integral
$\iint_{\Gamma \times \Gamma} r(x,y) \,
                   d\nu(x) \overline{d\nu(y)}.$
\end{corollary}
The following theorem expresses $\mu_{can}$ in terms of the resistance function:
\begin{theorem}\cite[Theorem 14.1]{BRh}\label{ncan}
The measure $\nn{can}(x)=\frac{1}{2}\DD{x}r(x,p)+\dd{p}(x)$ is of
total mass $1$ on $\ga$, which is independent of $p \in \ga$.
\end{theorem}
It is shown in \cite{CR} that as a function of three variables, on each edge $j_{x}(p,q)$ is a quadratic function of
$p$, $q$, $x$ and possibly with linear terms in $|x-p|$, $|x-q|$, $|p-q|$ if some of $p$, $q$, $x$ belong to the same edge.
These can be used to show that $\jj{x}{p}{q}$ is differentiable for $x \in \ga \backslash \big( \{p, q \} \cup \vv{\ga}\big)$.
Moreover, we have $\jj{x}{p}{q} \in Zh(\ga)$ for each $p$, $q$ and $x$ in $\ga$.

For any $x$, $p$ and $q$ in $\ga$, we can transform $\ga$ to an
$Y$-shaped graph with the same resistances between $x$, $p$, and
$q$ as in $\ga$ by applying a sequence of circuit reductions. The resulting graph is shown in Figure \ref{fig xpq1new2}, with the corresponding voltage values on each segment.
\begin{figure}
\centering
\includegraphics[scale=0.7]{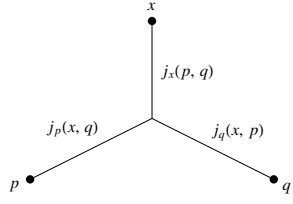} \caption{Circuit reduction with reference to $3$ points $x$, $p$ and $q$.} \label{fig xpq1new2}
\end{figure}
Then by \figref{fig xpq1new2}, we have
\begin{equation}\label{eqn1.1}
\begin{split}
\\r(p,x) = \jj{p}{x}{q}+\jj{x}{p}{q}, \; \, r(q,x) = \jj{q}{x}{p}+\jj{x}{p}{q},
\; \, r(p,q) = \jj{q}{x}{p}+\jj{p}{x}{q},
\end{split}
\end{equation}
so
\begin{equation}\label{eqn1.2}
\begin{split}
\DD{x}r(p,x) = \DD{x}\jj{p}{x}{q} & +\DD{x}\jj{x}{p}{q},  \qquad \quad
\DD{x}r(q,x) =\DD{x}\jj{q}{x}{p}+\DD{x}\jj{x}{p}{q},
\\ &\DD{x}r(p,q) =\DD{x}\jj{q}{x}{p}+\DD{x}\jj{p}{x}{q}=0.
\end{split}
\end{equation}

Using these formulas, we can express $\mu_{can}$ in terms of the voltage function in the following way:
\begin{theorem}\label{lemmadcanjpq}
For any $p,q \in \ga$,
$ \quad 2\nn{can}(x)=\DD{x}\jj{x}{p}{q}+\dd{q}(x)+\dd{p}(x)$.
\end{theorem}
\begin{proof}By \propref{prop cordjpq} and \eqnref{eqn1.2},
\begin{equation}\label{eqn lem can jpq}
\begin{split}
\DD{x}r(x,p) = \DD{x}\jj{x}{p}{q}+\dd{q}(x)-\dd{p}(x).
\end{split}
\end{equation}
Hence, the result follows from \thmref{ncan}.
\end{proof}
Let $e_i \in \ee{\ga}$ be an edge for which $\ga-e_i$ is connected, and let $\li$ be the length of $e_i$.
Suppose $\pp$ and $\qq$ are the end points of $e_i$, and $p \in \ga-e_i$. By applying circuit reductions,
$\ga-e_i$ can be transformed into a $Y$-shaped graph with the same resistances between $\pp$, $\qq$, and $p$ as in $\ga-e_i$.
The resulting graph is shown by the first graph in \figref{fig edgedelete5}, with the corresponding voltage values on each segment, where $\hj{x}{y}{z}$ is the voltage function in $\ga-e_i$. Since $\ga-e_i$ has such a circuit reduction, $\ga$ has the circuit reduction shown in the second graph in \figref{fig edgedelete5}. Throughout this paper, we will use the following notation:
$R_{a_i,p} := \hj{\pp}{p}{\qq}$, $R_{b_i,p} := \hj{\qq}{\pp}{p}$, $R_{c_i,p} := \hj{p}{\pp}{\qq}$, and $\ri$ is the resistance
between $\pp$ and $\qq$ in $\ga-e_i$. Note that $R_{a_i,p}+R_{b_i,p}=\ri$ for each $p \in \ga$. When $\ga-e_i$ is not connected, we set $R_{b_i,p}=\ri=\infty$ and $R_{a_i,p}=0$ if $p$ belongs to the component of $\ga-e_i$
containing $\pp$, and we set $R_{a_i,p}=\ri=\infty$ and $R_{b_i,p}=0$ if $p$ belongs to the component of $\ga-e_i$
containing $\qq$.
\begin{figure}
\centering
\includegraphics[scale=0.7]{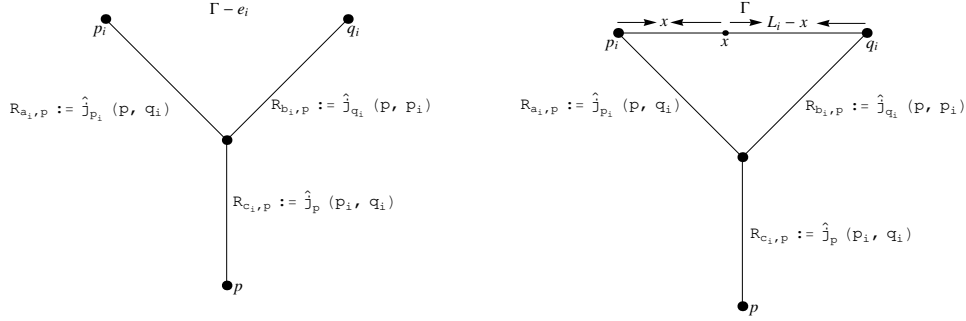} \caption{Circuit reduction of $\ga-e_i$ with reference to $\pp$, $\qq$ and $p$.} \label{fig edgedelete5}
\end{figure}

Another description of the tau constant is given below.
\begin{proposition}\cite{REU} \label{proptau}
Let $\Gamma$ be a metrized graph, and let $L_i$ be the length of the
edge $e_{i}$, for $i \in \{1,2, \dots, e\}$.
Using the notation above,
if we fix a vertex $p$ we have
\[
\ta{\ga} = \frac{1}{12} \sum_{e_i \in \ga} \left(
\frac{\li^3+3\li(R_{a_{i},p}-R_{b_{i},p})^2}{(\li+\ri)^2}\right ).
\]
Here, if $\ga-e_i$ is not connected, i.e. $\ri$ is infinite, the
summand corresponding to $e_i$ should be replaced by $3\li$, its limit as $\ri \longrightarrow
\infty$.
\end{proposition}
\begin{proof}
We start by fixing a vertex point $p \in \vv{\ga}$. By applying
circuit reductions, we can transform $\ga$ to the graph as in the second graph
in \figref{fig edgedelete5} when $x \in e_i$. Then, applying parallel reduction gives
$$r(x,p)=\frac{(x+R_{a_{i},p})(\li-x+R_{b_{i},p})}{\li+\ri}+R_{c_{i},p}.$$
Thus,
\begin{equation}\label{eqnproptau1}
\ddx r(x,p)=\begin{cases}\frac{\li-2x+R_{b_{i},p}-R_{a_{i},p}}{\li+\ri}, &  \, \text{if } \, \ga-e_i \, \text{is connected}, \\
\epsilon, &  \,\text{if } \, \ga-e_i \, \text{is disconnected},
\end{cases}
\end{equation}
where $\epsilon$ is $+1$ or $-1$, depending on which component of $\ga-e_i$ the point $p$ belongs to.

By \lemref{lemtauformula},
\begin{equation}\label{eqnproptau2}
\begin{split}
\tg =\frac{1}{4}\int_{\ga}\big(\ddx r(x,p)\big)^2dx
=\frac{1}{4}\sum_{e_i \in \, \ee{\ga}}\int_{e_i}\big(\ddx
r(x,p)\big)^2dx.
\end{split}
\end{equation}
Computing the integral after substituting \eqnref{eqnproptau1} into
\eqnref{eqnproptau2} gives the result.
\end{proof}
%

Chinburg and Rumely showed in \cite[page 26]{CR} that
\begin{equation}\label{eqn genus}
\sum_{e_i \in \ee{\ga}}\frac{\li}{\li +\ri}=g, \quad \text{equivalently } \sum_{e_i \in \ee{\ga}}\frac{\ri}{\li +\ri}=v-1.
\end{equation}
\begin{remark}\noindent{\bf Valence Property of $\tg$}\label{remvalence}
Let $\ga$ be any metrized graph with resistance function $r(x,y)$.
The formula for $\tg$ given in \propref{proptau} is independent of the chosen
point $p \in \vv{\ga}$, where $\vv{\ga}$ is the specified vertex set. In particular, enlarging $\vv{\ga}$ by including points $p \in \ga$ with $\va(p)=2$ does not change $\tg$. Thus, $\tg$ depends only
on the topology and the edge length distribution of the metrized
graph $\ga$.
\end{remark}
Let $\ga$ be a metrized graph with $e$ edges. Then
$\displaystyle \sum_{p \in \vv{\ga}} \va(p) = 2e$.
This is the ``Handshaking Lemma" of graph theory.
\begin{remark}\label{rem2term}
By \propref{proptau}, for any $p$ and $q$ in $\vv{\ga}$,
 $$\sum_{e_i \in \, \ee{\ga}}\frac{\li(R_{a_{i},p}-R_{b_{i},p})^2}{(\li+\ri)^2}
=\sum_{e_i \in \, \ee{\ga}}\frac{\li(R_{a_{i},q}-R_{b_{i},q})^2}{(\li+\ri)^2}.$$
\end{remark}
Let $\ga$ be a graph and let $p \in \vv{\ga}$. If a vertex $p$ is an end point of an edge $e_i$, then we write $e_i \sim p$. Since one of $R_{a_{i},p}$ and $R_{b_{i},p}$ is $0$ and the other is $\ri$
for every edge $e_i \sim p$,
\begin{equation}\label{eqn2term1}
\sum_{e_i \in \, \ee{\ga}}\frac{\li(R_{a_{i},p}-R_{b_{i},p})^2}{(\li+\ri)^2}
=\sum_{\substack{e_i \sim p\\ e_i \in \, \ee{\ga}}}\frac{\li \ri^2}{(\li+\ri)^2}
+\sum_{\substack{e_i \not\sim p\\ e_i \in \, \ee{\ga}}}\frac{\li(R_{a_{i},p}-R_{b_{i},p})^2}{(\li+\ri)^2}.
\end{equation}
\begin{lemma}\label{lem2term}
Let $\ga$ be a graph and $p \in \vv{\ga}$.
Then
\begin{equation*}
\begin{split}
\sum_{e_i \in \, \ee{\ga}}\frac{\li(R_{a_{i},p}-R_{b_{i},p})^2}{(\li+\ri)^2}
= \frac{2}{v}\sum_{e_i \in \, \ee{\ga}}\frac{\li\ri^2}{(\li+\ri)^2}
+ \frac{1}{v}\sum_{p \in \, \vv{\ga}}\Bigg(\sum_{\substack{e_i \not\sim p\\ e_i \in \,
\ee{\ga}}}\frac{\li(R_{a_{i},p}-R_{b_{i},p})^2}{(\li+\ri)^2} \Bigg).
\end{split}
\end{equation*}
\end{lemma}
\begin{proof}
By \remref{rem2term}, summing up \eqnref{eqn2term1} over all $p \in \vv{\ga}$ and dividing by $v=\#(\vv{\ga})$ gives
\begin{equation}\label{eqn2term2}
\begin{split}
\sum_{e_i \in \, \ee{\ga}}\frac{\li(R_{a_{i},p}-R_{b_{i},p})^2}{(\li+\ri)^2}
& =\frac{1}{v}\sum_{p \in \, \vv{\ga}}\Bigg(\sum_{\substack{e_i \sim p\\ e_i \in \, \ee{\ga}}}\frac{\li \ri^2}{(\li+\ri)^2}\Bigg)
\\ & \qquad +\frac{1}{v}\sum_{p \in \, \vv{\ga}}\Bigg(\sum_{\substack{e_i \not\sim p\\ e_i \in \,
\ee{\ga}}}\frac{\li(R_{a_{i},p}-R_{b_{i},p})^2}{(\li+\ri)^2} \Bigg).
\end{split}
\end{equation}
Each edge that is not a self loop is incident on exactly two vertices. On the other hand, $\ri=R_{a_{i},p}=R_{b_{i},p}=0$ for an edge $e_i$ that is a self loop.
Thus, the result follows from \eqnref{eqn2term2}.
\end{proof}
It was shown in \cite[Equation 14.3]{BRh} that for a metrized graph
$\Gamma$ with $e$ edges, we have
\begin{equation} \label{FMM1}
\frac{1}{16e} \ell(\Gamma) \ \leq \ \tau(\Gamma) \ \leq \
\frac{1}{4} \ell(\Gamma) \ ,
\end{equation}
with equality in the upper bound if and only if $\Gamma$ is a tree.
However, the lower bound is not sharp, and
Baker and Rumely posed the following lower bound
conjecture:
\begin{conjecture}\cite{BRh}\label{TauBound}
There is a universal constant $C>0$
such that for all metrized graphs $\Gamma$,
$$\tau(\Gamma) \ \geq \ C \cdot \ell(\Gamma)\ .$$
\end{conjecture}
\begin{remark}\label{rem C=1/108}
As can be seen from the examples and the cases we consider later in this paper, there is good evidence that $C=\frac{1}{108}$.
\end{remark}
\begin{remark}\cite{BRh}\label{rem tau scale-idependence}
If we multiply all lengths on $\Gamma$ by a
positive constant $c$, we obtain a graph $\Gamma'$ of
total length $c \cdot \ell(\Gamma)$. Then $\tau(\Gamma')
= c \cdot \tau(\Gamma)$. This will be called the \textit{scale-independence} of the tau constant.
By this property, to prove \conjref{TauBound}, it is enough to consider metrized graphs with total length $1$.
\end{remark}
The following proposition gives an explicit formula for the tau constant for complete graphs, for which \conjref{TauBound} holds with $C=\frac{23}{500}$.
\begin{proposition}\label{prop tau formula for complete graph}
Let $\ga$ be a complete graph on $v$ vertices with equal edge lengths. Suppose $v \geq 2$. Then we have
\begin{equation}\label{eqn prop tau}
\begin{split}
\tg =\Big( \frac{1}{12}\big(1-\frac{2}{v}\big)^2+\frac{2}{v^3}\Big)\ell (\ga).
\end{split}
\end{equation}
In particular, $\tg \geq \frac{23}{500} \elg$, with equality when $v=5$.
\end{proposition}
\begin{proof}
Let $\ga$ be a complete graph on $v$ vertices. If $v=2$, then $\ga$ contains only one edge $e_1$ of length $L_1$, i.e. $\ga$ is a line segment. In this case, $R_1$ is infinite. Therefore, $\tg=\frac{L_1}{4}$ by \propref{proptau}, which coincides with \eqnref{eqn prop tau}. Suppose $v \geq 3$. Then the valence of any vertex is $v-1$, so by basic graph theory $e=\frac{v(v-1)}{2}$, and $g=\frac{(v-1)(v-2)}{2}$.
Since all edge lengths are equal, $\li=\frac{\ell (\ga)}{e}$ for each edge $e_i \in \ee{\ga}$. By the symmetry of the graph, we have $\ri=R_j$ for any two edges $e_i$ and $e_j$ of $\ga$.
Thus \eqnref{eqn genus} implies that $\ri=\frac{2 \li}{v-2}$ for each edge $e_i$. Moreover, by the symmetry of the graph again, $r(p,q)=\frac{\li \ri}{\li+\ri}$ for all distinct $p$, $q \in \vv{\ga}$. Again by the symmetry and the fact that $R_{a_{i},p}+R_{b_{i},p}=\ri$, we have $R_{a_{i},p}=R_{b_{i},p}=\frac{\ri}{2}$ for each edge $e_i$ with end points different from $p$. Substituting these values into the formula for $\tg$ given in \propref{proptau} and using \lemref{lem2term} gives the equality. The inequality $\tg \geq \frac{23}{500} \elg$ now follows by elementary calculus.
\end{proof}
For a circle graph, \conjref{TauBound} holds with $C=\frac{1}{12}$.
\begin{corollary}\label{cor tau formula for circle}
Let $\ga$ be a circle graph. Then we have
$\tg=\frac{\ell (\ga)}{12}.$
\end{corollary}
\begin{proof}
A circle graph can be considered as a complete graph on $3$ vertices. The vertices are of valence two, so by the valence property of $\ga$, edge length distribution does not effect the tau constant of $\ga$. If we position the vertices equally spaced on $\ga$, we can apply \propref{prop tau formula for complete graph} with $v=3$.
\end{proof}
The following theorem is frequently needed in computations related to
the tau constant. It is also interesting in its own right.
\begin{theorem}\label{thmjpq2njpq}
For any p, q $\in \ga$ and $-1 < n \in \RR$,
\begin{equation*}
\int_{\ga}(\frac{d}{dx}\jj{p}{x}{q})^2\jj{p}{x}{q}^ndx =
\frac{1}{n+1}r(p,q)^{n+1} .
\end{equation*}
\end{theorem}
\begin{proof}
Note that $\jj{p}{x}{q}^{n+1} \in \Zh(\Gamma)$ when $-1 < n \in \RR$.
\begin{equation*}
\begin{split}
(n+1)\int_{\ga}(\frac{d}{dx}\jj{p}{x}{q})^2\jj{p}{x}{q}^ndx &=
\int_{\ga}\frac{d}{dx}\jj{p}{x}{q}\frac{d}{dx}(\jj{p}{x}{q}^{n+1})dx
\\ &= \int_{\ga}\jj{p}{x}{q}^{n+1}\DD{x}\jj{p}{x}{q}, \quad \text{by
\propref{prop Greens identity}}
\\ & = \int_{\ga}\jj{p}{x}{q}^{n+1}(\dd{q}(x)-\dd{p}(x)).
\end{split}
\end{equation*}
Then the result follows from the properties of the voltage function.
\end{proof}
We isolate the cases $n=0, \, 1$, and $2$, since we will use them later
on.
\begin{corollary}\label{corjrpq}
For any $p$ and $q$ in $\ga$,
$$
\int_{\ga}(\frac{d}{dx}\jj{p}{x}{q})^2dx = r(p,q),
\quad  \quad \int_{\ga}(\frac{d}{dx}\jj{p}{x}{q})^2\jj{p}{x}{q}dx =
\frac{1}{2}r(p,q)^2 \quad \text{and} $$
$$\int_{\ga}(\frac{d}{dx}\jj{p}{x}{q})^2\jj{p}{x}{q}^2dx =
\frac{1}{3}r(p,q)^3.
$$
\end{corollary}
\begin{lemma}\label{lemorthogonality}
For any $p,q,x \in \ga$,
\begin{equation*}
\begin{split}
 \int_{\ga}\ddx\jj{x}{p}{q}\ddx\jj{p}{x}{q}dx =\int_{\ga}\jj{p}{x}{q}\DD{x}\jj{x}{p}{q}
=\int_{\ga}\jj{x}{p}{q}\DD{x}\jj{p}{x}{q}=0.
\end{split}
\end{equation*}
\end{lemma}
\begin{proof}
Since  $\DD{x}$ is a self-adjoint operator (see \propref{prop Greens identity}),
$$\int_{\ga}\jj{p}{x}{q}\DD{x}\jj{x}{p}{q}=\int_{\ga}\jj{x}{p}{q}\DD{x}\jj{p}{x}{q}=\jj{p}{p}{q}-\jj{q}{p}{q}=0,$$
where the second equality is by \propref{prop cordjpq}. Also, by the Green's identity (see \propref{prop Greens identity}),
$\int_{\ga}\jj{x}{p}{q}\DD{x}\jj{p}{x}{q}=\int_{\ga}\ddx\jj{p}{x}{q}\ddx\jj{x}{p}{q}dx.$
This completes the proof.
\end{proof}
Now we are ready to express the tau constant in terms of the voltage
function.
\begin{theorem}\label{thmbasic}
For any $p,q \in \ga$,
$
\tg = \frac{1}{4}\int_{\ga}(\frac{d}{dx}\jj{x}{p}{q})^2dx +
\frac{1}{4}r(p,q).
$
\end{theorem}
\begin{proof}
For any $p, \, q \in \ga$, we have
\begin{equation*}
\begin{split}
4 \tg &=\int_{\ga}\big(\frac{\partial}{\partial x} r(p,x) \big)^2dx, \quad \text{by \lemref{lemtauformula};}
\\ &=\int_{\ga}r(p,x)\DD{x}r(p,x), \quad \text{by the Green's identity;}
\\ &=\int_{\ga}r(p,x) \big(\DD{x}\jj{x}{p}{q}+\dd{q}(x)-\dd{p}(x) \big), \quad \text{by \eqnref{eqn lem can jpq};}
\\ &=\int_{\ga}r(p,x)\DD{x}\jj{x}{p}{q} +r(p,q), \quad \text{since $r(p,p)=0$;}
\\ &=\int_{\ga}(\jj{x}{p}{q}+\jj{p}{x}{q})\DD{x}\jj{x}{p}{q}
+r(p,q), \quad \text{by \eqnref{eqn1.1};}
\\ &=\int_{\ga}(\frac{d}{dx}\jj{x}{p}{q})^2dx
+\int_{\ga}\jj{x}{p}{q}\DD{x}\jj{p}{x}{q}+r(p,q),
\quad \text{by \propref{prop Greens identity};}
\\ &=\int_{\ga}(\frac{d}{dx}\jj{x}{p}{q})^2dx+r(p,q), \quad \text{by \lemref{lemorthogonality}}.
%
\end{split}
\end{equation*}
This is what we wanted to show.
\end{proof}
Since $\jj{x}{p}{p}=r(p,x)$ and $r(p,p)=0$, \lemref{lemtauformula} is the special case of
\thmref{thmbasic} with $q=p$.

Suppose $\ga$ is a graph which is the union of two subgraphs $\ga_1$ and $\ga_2$,
i.e., $\ga=\ga_1 \cup \ga_2$. If $\ga_1$ and $\ga_2$ intersect in a single point $p$, i.e.,
$\ga_1 \cap \ga_2= \{ p \}$, then by circuit theory (see
also \cite[Theorem 9 (ii)]{BF}) we have $r(x,y)=r(x,p)+r(p,y)$ for
each $x \in \ga_1$ and $y \in \ga_2$. By using this fact and
\corref{lemtauformula}, we obtain $\ta{\ga_1 \cup
\ga_2}=\ta{\ga_1}+\ta{\ga_2}$, which we call the ``\textit{additive
property}" of the tau constant. It was initially noted in \cite{REU}.

The following corollary of \thmref{thmbasic} was given in
\cite[Equation 14.3]{BRh}.
\begin{corollary}\label{lemtauformula2}
Let $\ga$ be a tree, i.e. a graph without cycles. Then,
$\tg=\frac{\ell(\ga)}{4}$.
\end{corollary}
\begin{proof}
First we note that for a line segment $\beta$ with end points $p$
and $q$, we have that $r(p,q)=\ell(\beta)$. 
It is clear by circuit theory that $\jj{x}{p}{q}=0$ for any $x \in \beta$, where
$\jj{x}{y}{z}$ is the voltage function on $\beta$. Therefore,
$\ta{\beta}=\frac{\ell(\beta)}{4}$ by \thmref{thmbasic}. Hence the result follows for any
tree graph by applying the additive property whenever it is needed.
\end{proof}
Thus, \conjref{TauBound} holds with $C=\frac{1}{4}$ for a tree graph.

\begin{corollary}\label{coradd}
Let $\ga$ be a metrized graph, and let $E_1(\ga)=\{e_i \in \ee{\ga}| \text{$e_i$ is a bridge}\}.$
Suppose $\overline{\ga}$ is the metrized graph obtained from $\ga$ by contracting
edges in $E_1(\ga)$ to their end points. Then
$\tg=\ta{\overline{\ga}}+\frac{\ell(\ga)-\ell(\overline{\ga})}{4}$.
\end{corollary}
\begin{proof}
If $E_1(\ga)\not=\emptyset$, we successively apply the additive property of the tau constant and \corref{lemtauformula2} to obtain the result.
\end{proof}
By \corref{coradd}, to prove \conjref{TauBound}, it is enough to prove it for bridgeless graphs.
\begin{theorem}[Baker]\label{thmeqlength}
Suppose all edge lengths in a metrized graph $\ga$ with $\ell(\ga)=1$  are
equal, i.e., of length $\frac{1}{e}$. Then
$ \tg \geq \frac{1}{12}(\frac{g}{e})^2$.
In particular, \conjref{TauBound} holds with
$ C= \frac{1}{108}$ if we also have $\va(p) \geq 3$ for each vertex $p \in \vv{\ga}$.
\end{theorem}
\begin{proof}
By \corref{coradd}, the scale-independence and the additive
properties of $\tg$, it will be enough to prove the result for a
graph $\ga$ that does not have any edge whose removal disconnects
it. Applying the Cauchy-Schwarz inequality to the second part of the equality
$\displaystyle \sum_{e_i \in \ee{\ga}}
\frac{\li^3}{(L_{i}+R_{i})^2} = \sum_{e_i \in \ee{\ga}} \frac{\li^3}{(L_{i}+R_{i})^2} \sum_{e_i
\in\ee{\ga}}\li$ gives
\begin{equation}\label{eqnthmcauchy1}
\begin{split}
\sum_{e_i \in \ee{\ga}}
\frac{\li^3}{(L_{i}+R_{i})^2} \geq \Big(\sum_{e_i \in \ee{\ga}}
\frac{\li^2}{L_{i}+R_{i}}\Big)^2.
\end{split}
\end{equation}
We have
\begin{equation*}\label{eqneqlength}
\begin{split}
\tg &\geq \frac{1}{12}\sum_{e_i \in \ee{\ga}}
\frac{\li^3}{(L_{i}+R_{i})^2}, \quad \text{by \propref{proptau}};
\\ &\geq \frac{1}{12}\Big(\sum_{e_i \in \ee{\ga}}
\frac{\li^2}{L_{i}+R_{i}}\Big)^2, \quad \text{by
\eqnref{eqnthmcauchy1}};
\\ &=\frac{1}{12}\Big(\frac{1}{e}\sum_{e_i \in \ee{\ga}}
\frac{\li}{L_{i}+R_{i}}\Big)^2, \quad \text{since all edge lengths
are equal;}
\\ &=\frac{1}{12}(\frac{g}{e})^2, \quad \text{by \eqnref{eqn genus}}.
\end{split}
\end{equation*}
This proves the first part.
If $\va(p) \geq 3$ for each $p \in \vv{\ga}$, then we have $e
\geq \frac{3}{2} v $ by basic properties of connected graphs.
Thus $g=e-v+1 \geq e-\frac{2}{3}e+1 \geq \frac{e}{3}$. Using this inequality along with the first part gives the last part.
\end{proof}
In the next theorem, we show that \conjref{TauBound} holds for another large class of graphs with $C=\frac{1}{48}$. First, we recall Jensen's Inequality:

For any integer $n\geq 2$, let $a_i \in (c,d)$, an interval in $\RR$, and $b_i \geq 0$ for all $i=1, \dots, n$.
If $f$ is a convex function on the interval $(c,d)$, then
$$f\Big( \frac{\sum_{i=1}^n b_ia_i}{\sum_{i=1}^n b_i}\Big)\leq \frac{\sum_{i=1}^n b_if(a_i)}{\sum_{i=1}^n b_i}.$$
The inequality is reversed, if $f$ is a concave function on $(c,d)$.
\begin{theorem}\label{thmcorineqsumR4}
Let $\ga$ be a graph with $\ell(\ga)=1$ and let $\li$, $\ri$ be as
before. Then we have
$\tg \geq \frac{1}{12}\frac{1}{\big(1+\sum_{e_i \in \ee{\ga}}\ri\big)^2}.$
In particular, if any pair of vertices $\pp$ and $\qq$ that are end
points of an edge are joined by at least two edges,
we have
$\tg \geq \frac{1}{48}.$
\end{theorem}
\begin{proof}
Let $b_i=L_{i}$, $a_{i}=\frac{\li+\ri}{\li}$, and
$f(x)=\frac{1}{x}$ on $(0,\infty)$. Then applying Jensen's
inequality and using the assumption that $\sum b_i= \ell(\ga)=1$, we obtain
the following inequality:
\begin{equation}\label{eqn sumR}
\begin{split}
 \sum_{e_i \in \ee{\ga}}\frac{\li^2}{L_{i}+R_{i}} \geq \frac{1}{1+\sum_{e_i \in \ee{\ga}}\ri}.
\end{split}
\end{equation}
Then the first part follows from \propref{proptau}, \eqnref{eqnthmcauchy1}, and \eqnref{eqn sumR}.
Under the assumptions of the second part,
we obtain $\sum_{e_i \in \ee{\ga}}\ri \leq \sum_{e_i \in \ee{\ga}}\li=1$
by applying parallel circuit reduction.
This yields the second part.
\end{proof}
Additional proofs of \eqnref{eqn sumR} can be found in \cite[page 50]{C1}.
\begin{theorem}\label{thm2term}
Let $\ga$ be a metrized graph with $\elg=1$. Then we have
$$\sum_{e_i \in \, \ee{\ga}}\frac{\li\ri^2}{(\li+\ri)^2} \geq \Big(\sum_{e_i \in \, \ee{\ga}}\frac{\li\ri}{\li+\ri}\Big)^2.$$
\end{theorem}
\begin{proof}
We have $\ell(\ga)=1$. Hence, by Cauchy-Schwarz inequality
$$\sum_{e_i \in \ee{\ga}} \frac{\li \ri^2}{(L_{i}+R_{i})^2}= \sum_{e_i \in
\ee{\ga}} \frac{\li \ri^2}{(L_{i}+R_{i})^2} \sum_{e_i \in\ee{\ga}}\li \geq \Big(\sum_{e_i \in \ee{\ga}} \frac{\li \ri}{L_{i}+R_{i}}\Big)^2.$$
\end{proof}
The following theorem improves \thmref{thmeqlength} slightly:
\begin{theorem}\label{thmeqlength2}
Suppose all edge lengths in a graph $\ga$ with $\ell(\ga)=1$  are
equal, i.e., of length $\frac{1}{e}$. Then
$ \tg \geq \frac{1}{12}(\frac{g}{e})^2+\frac{1}{2v}(\frac{v-1}{e})^2$.
\end{theorem}
\begin{proof}
It follows from \lemref{lem2term} and \thmref{thm2term}  that
$$\sum_{e_i \in \, \ee{\ga}}\frac{\li(R_{a_{i},p}-R_{b_{i},p})^2}{(\li+\ri)^2}
\geq \frac{2}{v}\big( \sum_{e_i \in \, \ee{\ga}}\frac{\li\ri}{\li+\ri} \big)^2.$$
Since $\li=\frac{1}{e}$ for each edge $e_i$, $\sum_{e_i \in \, \ee{\ga}}\frac{\li\ri}{\li+\ri}=
\frac{1}{e}\sum_{e_i \in \, \ee{\ga}}\frac{\ri}{\li+\ri}=\frac{v-1}{e}$ by using \eqnref{eqn genus}.
Therefore, the result follows from \propref{proptau} and the proof of \thmref{thmeqlength}.
\end{proof}
In the next section, we will derive explicit formulas for the tau constants of certain graphs with multiple edges.
\section{The tau constants of metrized graphs with multiple edges}\label{section DA}
Let $\ga$ be an arbitrary graph; write $\ee{\ga}=\{e_1, e_2, \dots,
e_e\}$. As before, let $\li$ be the length of edge $e_i$. Let
$\ga^{DA,n}$, for a positive integer $n \geq 2$, be the graph
obtained from $\ga$ by replacing each edge $e_i \in \ee{\ga}$ by $n$
edges $e_{i,1}, e_{i,2}, \dots, e_{i,n} $ of equal lengths
$\frac{\li}{n}$. (Here DA stands for ``Double Adjusted".) Then,
$\vv{\ga}=\vv{\ga^{DA,n}}$ and $\ell(\ga)=\ell(\ga^{DA,n})$. We set
$\ga^{DA}:=\ga^{DA,2}$. The following observations will enable us to
compute $\ta{\ga^{DA,n}}$ in terms of $\tg$.

We will denote by $R_{j}(\ga)$ the resistance between end points of
an edge $e_j$ of a graph $\ga$  when the edge $e_j$ is deleted from
$\ga$.

Figure \ref{fig doubling} shows the edge replacement for an edge
when $n=4$.
\begin{figure}
\centering
\includegraphics[scale=0.6]{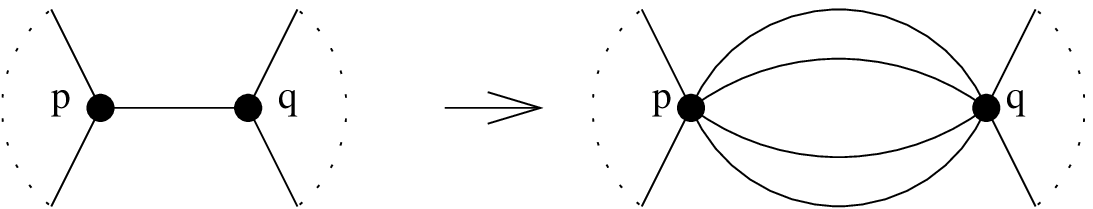} \caption{$\ga$ and
$\ga^{DA,4}$} \label{fig doubling}
\end{figure}
%
A graph with two vertices and $m$ edges connecting the vertices will be called a $m$-banana graph.
\begin{lemma}\label{lemparallel1} Let $\beta$ be a $m$-banana graph, as shown in Figure \ref{fig doubparred},
such that $\li=L$ for each $e_i \in \beta$. Let $r(x,y)$ be the
resistance function in $\beta$, and let $p$ and $q$ be the end points
of all edges. Then,
$r(p,q)=\frac{L}{m}.$
\end{lemma}
\begin{proof}
By parallel circuit reduction,
$\frac{1}{r(p,q)}=\sum_{k=1}^m\frac{1}{L}=\frac{m}{L}$. Hence, the
result follows.
\end{proof}
\begin{figure}
\centering
\includegraphics[scale=0.6]{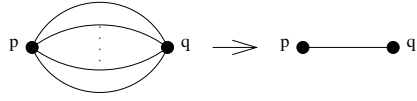} \caption{Circuit reduction
for a banana graph.} \label{fig doubparred}
\end{figure}
\begin{remark}\label{remresproportional}
If we divide each edge length of a graph $\ga$, with resistance
function $r(x,y)$, by a positive number $k$, we obtain a graph with
resistance function $\frac{r(x,y)}{k}$.
\end{remark}
\begin{corollary}\label{corparallel2}
Let $r(x,y)$ and $r^n(x,y)$ be the resistance functions in $\ga$ and
$\ga^{DA,n}$, respectively. Then, for any $p$ and $q$ $\in \vv{\ga}
$, $r^n(p,q)=\frac{r(p,q)}{n^2}.$
\end{corollary}
\begin{proof}
By using \lemref{lemparallel1}, every group of $n$ edges $e_{i,1},
e_{i,2}, \dots, e_{i,n} $, in $\ee{\ga^{DA,n}}$, corresponding to
edge $e_i \in \ee{\ga}$ can be transformed into an edge $e_i'$. When
completed,
 this process results in a graph which can also be obtained
from $\ga$ by dividing each edge length $\li$ by $n^2$. Therefore,
the result follows from \remref{remresproportional}.
\end{proof}
\begin{theorem}\label{thmdouble}
Let $\ga$ be any graph, and let $\ga^{DA,n}$ be the related graph
described before. Then
\[
\ta{\ga^{DA,n}}=\frac{\tg}{n^2} +
\frac{\ell(\ga)}{12}\big(\frac{n-1}{n}\big)^2 +
\frac{n-1}{6n^2}\sum_{e_i \in \, \ee{\ga}}\frac{\li^2}{\li+\ri}.
\]
\end{theorem}
\begin{proof}
Let $p$ be a fixed vertex in $\vv{\ga}=\vv{\ga^{DA,n}}$. Whenever $x
\in e_{i,j}$ for some $j \in \{1, 2, \dots, n\}$,  we can transform
the graph $\ga^{DA,n}$ to the graph as shown in Figure \ref{fig
doublethmp} by using \corref{corparallel2}, \corref{corparallel2}
and circuit reduction for $\ga-e_i$. (Here $R_{a_{i,p}}$,
$R_{b_{i,p}}$ and $R_{c_{i,p}}$ are as in
 \propref{proptau}
 and so
$R_{a_{i,p}}+R_{b_{i,p}}=\ri$.)
\begin{figure}
\centering
\includegraphics[scale=0.6]{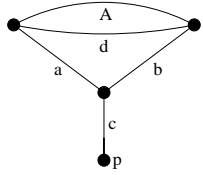} \caption{Circuit reduction
for $\ga^{DA,n}$ with reference to an edge and a point.} \label{fig
doublethmp}
\end{figure}

In Figure \ref{fig doublethmp}, we have
$a=\frac{R_{a_{i,p}}(\ga)}{n^2}, \,$
$b=\frac{R_{b_{i,p}}(\ga)}{n^2}, \,$
$c=\frac{R_{c_{i,p}}(\ga)}{n^2}$, $A$ is the edge $e_{i,j}$ of
length $\frac{\li}{n}$, and $d=\frac{\li}{n(n-1)}.$ Then, by using
a Delta-Wye transformation followed by parallel circuit reduction, we
derive the formula below for the effective resistance between a
point $x \in e_{i,j}$ and $p$, which will be denoted by $r^n(x,p)$.
\begin{equation}\label{eqndouble1}
r^n(x,p)= \frac{\big(x+\frac{ad}{a+b+d}\big)\big(\frac{\li}{n}-x+\frac{db}{a+b+d} \big)}
{\frac{\li}{n}+\frac{ad+db}{a+b+d}} +\frac{ab}{a+b+d}+c.
\end{equation}
By using \corref{lemtauformula},
\begin{equation}\label{eqndouble2}
\begin{split}
\ta{\ga^{DA,n}} &= \frac{1}{4} \int_{\ga^{DA,n}} \left( \ddx r(x,y)
\right)^2 dx .
\\ &= \frac{1}{4} \sum_{e_{i,j} \in \ee{\ga^{DA,n}}} \int_{e_{i,j}} \left( \ddx r(x,y)
\right)^2 dx .
\\ &= \frac{n}{4} \sum_{e_{i} \in \ee{\ga}} \int_{0}^{\frac{\li}{n}} \left( \ddx r(x,y)
\right)^2 dx, \quad \text{by symmetry within multiple edges.}
\end{split}
\end{equation}
This integral was computed using Maple, after substituting the derivative of
\eqnref{eqndouble1} and the values of $a$, $b$ and $d$ as above into
\eqnref{eqndouble2}. Let
\begin{equation}\label{eqndouble200}
\begin{split}
I_i:=&\int_{0}^{\frac{\li}{n}} \left( \ddx r(x,y)\right)^2 dx,\quad
\text{and let}
\\ J_i:= &\frac{\li}{12}(\frac{n-1}{n})^2+\frac{n-1}{6n^2}\frac{\li^2}{\li+\ri}+\frac{1}{12n^2}
\frac{\li^3+3\li(R_{a_{i,p}}-R_{b_{i,p}})^2} {(\li+\ri)^2}.
\end{split}
\end{equation}
Then, via Maple, $\frac{n}{4}I_i=J_i$. Inserting this into
\eqnref{eqndouble2} and using \propref{proptau}, we see that
$\ta{\ga^{DA,n}}=\sum_{e_{i} \in \ee{\ga}} J_i$. This yields the
theorem.
\end{proof}
In \secref{section general DA}, we will give a far-reaching
generalization of \thmref{thmdouble}.
\begin{corollary}\label{cordoubleN=2}
Let $\ga$ be a graph. Then,
\begin{equation*}
\begin{split}
\ta{\ga^{DA}} &=\frac{\tg}{4} + \frac{\ell(\ga)}{48} + \frac{1}{24}\sum_{e_i \in \, \ee{\ga}}\frac{\li^2}{\li+\ri}.
\\ \ta{\ga^{DA,3}} &=\frac{\tg}{9} +
\frac{\ell(\ga)}{27} + \frac{1}{27}\sum_{e_i \in \,
\ee{\ga}}\frac{\li^2}{\li+\ri}.
\end{split}
\end{equation*}
\end{corollary}
\begin{proof}
Setting $n=2$ and $n=3$ in \thmref{thmdouble} gives the equalities.
\end{proof}
\begin{figure}
\centering
\includegraphics[scale=0.4]{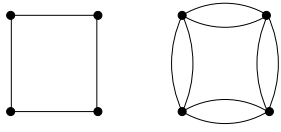} \caption{$\ga \rightarrow
\ga^{DA}$, doubling the edges.} \label{fig double2}
\end{figure}
\begin{corollary}\label{cordoublebanana}
Let $\ga$ be a banana graph with $n\geq 1$ edges that have equal
length. Then,
$$\ta{\ga}=\frac{\ell(\ga)}{4n^2} + \frac{\ell(\ga)}{12}\big(\frac{n-1}{n}\big)^2=\frac{\ell(\ga)}{12}\frac{n^2-2n+4}{n^2}
\geq \frac{\elg}{16}.$$
\end{corollary}
\begin{proof}
 Let $\beta$ be a line segment of length $\ell(\ga)$. Since $R_1(\beta)=\infty$,
$\ta{\beta^{DA,n}}=\frac{\ta{\beta}}{n^2} +
\frac{\ell(\beta)}{12}\big(\frac{n-1}{n}\big)^2+0$ by
\thmref{thmdouble}. On the other hand, we have $\beta^{DA,n}=\ga$,
$\ell(\beta)=\ell(\ga)$, and $\ta{\beta}=\frac{\ell(\beta)}{4}$
since $\beta$ is a tree. This gives the equalities we want to show, and the inequality follows by Calculus.
\end{proof}
By dividing each edge $e_i \in \ee{\ga}$ into $m$ equal subsegments and
considering the end points of the subsegments as new vertices, we obtain
a new graph which we denote by $\ga^m$. Note that $\ga$ and $\ga^m$
have the same topology, and $\ell(\ga)=\ell(\ga^m)$, but
$\#(\ee{\ga^m})=m \cdot \#(\ee{\ga}) =m \cdot e$ and
$\#(\vv{\ga^m})=\#(\vv{\ga})+(m-1) \cdot \#(\ee{\ga})=v+(m-1)
\cdot e$. Figure \ref{fig mdivision2} shows an example when $\ga$ is
a line segment with end points $p$ and $q$, and $m=3$.
\begin{figure}
\centering
\includegraphics[scale=0.5]{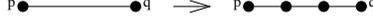} \caption{Division into $m=3$
equal parts.} \label{fig mdivision2}
\end{figure}
%

Suppose an edge $e_k \in \ee{\ga^m}$ has end points $p_k$ and $q_k$ that are in $\vv{\ga^m}$.
To avoid any potential misinterpretation, we will denote the length of $e_k$ by $L_k(\ga^m)$. Likewise, the
resistance between $p_k$ and $q_k$ in $\ga^m-e_k$ will be denoted by $R_k(\ga^m)$.
\begin{lemma}\label{lemdivision1}
Let $\ga$ be a graph, and $\ga^m$ be as defined. Then the following identities
hold:
\begin{equation*}
\begin{split}
& (i) \qquad  \sum_{e_k \in
\ee{\ga^m}}\frac{L_k(\ga^m)^2}{L_k(\ga^m)+R_k(\ga^m)}
=\frac{1}{m}\sum_{e_i \in \ee{\ga}}\frac{\li^2}{\li+\ri}.
\\ & (ii) \qquad  \sum_{e_k \in \ee{\ga^m}}\frac{L_k(\ga^m)^3}{(L_k(\ga^m)+R_k(\ga^m))^2}
=\frac{1}{m^2}\sum_{e_i \in \ee{\ga}}\frac{\li^3}{(\li+\ri)^2}.
\\ & (iii) \qquad  \sum_{e_k \in \ee{\ga^m}}\frac{L_k(\ga^m)R_k(\ga^m)}{L_k(\ga^m)+R_k(\ga^m)}
=\frac{m-1}{m}\ell(\ga)+\frac{1}{m}\sum_{e_i \in
\ee{\ga}}\frac{\li\ri}{\li+\ri}.
\end{split}
\end{equation*}
\end{lemma}
\begin{proof}
Proof of part $(i)$:
Note that subdivision of an edge in $\ee{\ga}$ results in $m$ edges in $\ee{\ga^m}$. If $e_k \in \ee{\ga^m}$ is one of the
edges corresponding to an edge $e_i \in \ee{\ga}$, then we have
$L_k(\ga^m)=\frac{\li}{m}$ and
$R_k(\ga^m)=\frac{m-1}{m}\li+\ri$.
Therefore, $L_k(\ga^m)+R_k(\ga^m)=\li+\ri$ giving
\begin{equation*}
\sum_{e_k \in
\ee{\ga^m}}\frac{L_k(\ga^m)^2}{L_k(\ga^m)+R_k(\ga^m)}
=\sum_{j=1}^m \Big(\frac{1}{m^2}\sum_{e_i
\in \ee{\ga}}\frac{\li^2}{\li+\ri} \Big)
=\frac{1}{m}\sum_{e_i \in \ee{\ga}}\frac{\li^2}{\li+\ri}.
\end{equation*}
The proofs of parts $(ii)$ and $(iii)$ follow by similar
calculations.
\end{proof}
\begin{theorem}\label{thmdoubledivision}
Let $\ga$ be a graph, and let $\ga^m$ be as above. Then,
$$\ta{(\ga^m)^{DA,n}}=\frac{\tg}{n^2} +
\frac{\ell(\ga)}{12}\big(\frac{n-1}{n}\big)^2 +
\frac{n-1}{6mn^2}\sum_{e_i \in \, \ee{\ga}}\frac{\li^2}{\li+\ri}.$$
\end{theorem}
\begin{proof}
Applying \thmref{thmdouble} to $\ga^m$ gives
\begin{equation}\label{eqndivision2}
\begin{split}
\ta{(\ga^m)^{DA,n}} =\frac{\ta{\ga^m}}{n^2} +
\frac{\ell(\ga^m)}{12}\big(\frac{n-1}{n}\big)^2
+\frac{n-1}{6n^2}\sum_{e_k \in \,
\ee{\ga^m}}\frac{L_k(\ga^m)^2}{L_k(\ga^m)+R_k(\ga^m)}.
\end{split}
\end{equation}
Since $\ell(\ga^m)=\ell(\ga)$ and $\ta{\ga^m}=\tg$,
the result follows from part $(i)$ of \lemref{lemdivision1}.
\end{proof}
\begin{example}
Let $\ga$ be the circle graph with one vertex, and let $\ga^m$ be as above (see also \figref{fig curvewdm2}). Since $\tg =\frac{\elg}{12}$ and
$\sum_{e_i \in \ee{\ga}}\frac{\li^2}{\li+\ri}=\elg$, we have $\ta{(\ga^m)^{DA,n}}=\big(\frac{(n-1)^2+1}{12 n^2} +\frac{n-1}{6 m n^2}\big)\elg$ by using \thmref{thmdoubledivision}. In particular, we have $\ta{(\ga^m)^{DA}}=\frac{1}{24}\elg +\frac{1}{24 m}\elg$.
\end{example}
\begin{figure}
\centering
\includegraphics[scale=0.35]{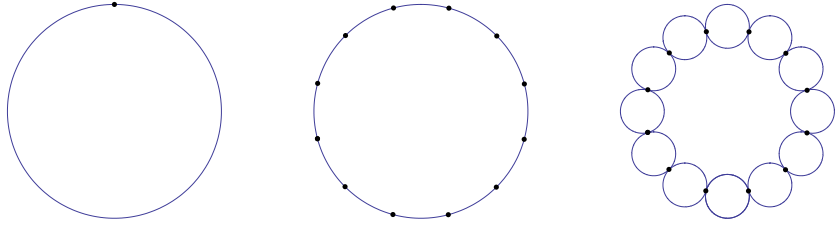} \caption{Circle graph, circle with multi vertices and the corresponding double graph.} \label{fig curvewdm2}
\end{figure}
\begin{lemma}\label{lemdivisione}
Let $\ga$ be a graph. The following identities hold:
\begin{equation*}
\begin{split}
& (i) \qquad  \ri(\ga^{DA,n})=\frac{1}{n}\frac{\li \ri}{(n\li+(n-1)\ri)}.
\\ & (ii) \qquad \sum_{e_i \in \,
\ee{\ga^{DA,n}}}\frac{\li(\ga^{DA,n})^2}{\li(\ga^{DA,n})+\ri(\ga^{DA,n})}
 = \frac{n-1}{n}\ell(\ga)+\frac{1}{n}\sum_{e_i \in \,
\ee{\ga}}\frac{\li^2}{\li+\ri}.
\end{split}
\end{equation*}
\end{lemma}
\begin{proof}The proof of $(i)$:
By the proof of \thmref{thmdouble} with its notation $a$, $b$, $d$,
\begin{equation*}
\begin{split}
\ri(\ga^{DA,n})=\frac{d(a+b)}{d+a+b}= \frac{\frac{\li}{n(n-1)}
\frac{\ri}{n^2}}{\frac{\li}{n(n-1)}+\frac{\ri}{n^2}}
=\frac{1}{n}\frac{\li \ri}{(n\li+(n-1)\ri)}.
\end{split}
\end{equation*}
The proof of $(ii)$: By using part $(i)$,
\begin{equation*}
\begin{split}
\sum_{e_i \in \,
\ee{\ga^{DA,n}}}\frac{\li(\ga^{DA,n})^2}{\li(\ga^{DA,n})+\ri(\ga^{DA,n})}
=n \sum_{e_i \in \,
\ee{\ga}}\frac{(\frac{\li}{n})^2}{\frac{\li}{n}+\frac{1}{n}\frac{\li\ri}{n\li+(n-1)\ri}}.
\end{split}
\end{equation*}
Then the result follows.
%
%
\end{proof}
\begin{theorem}\label{thmdoubleimp}
Let $\ga$ be a graph with $\ell(\ga)=1$. Suppose $\ta{\ga^{DA,n}}
\geq \frac{1}{108}\big(\frac{3n-2}{n}\big)^2$. Then $\tg \geq
\frac{1}{108}$.
\end{theorem}
\begin{proof}
By \thmref{thmdouble},
$\tg=n^2\ta{\ga^{DA,n}}-\frac{(n-1)^2}{12}-\frac{n-1}{6}\sum_{e_i \in \,
\ee{\ga}}\frac{\li^2}{\li+\ri}.$ On the other hand, by
the proof of \thmref{thmeqlength}
 $\tg \geq  \frac{1}{12}\Big(\sum_{e_i \in \,
\ee{\ga}}\frac{\li^2}{\li+\ri}\Big)^2.$ Let $x=\sum_{e_i \in \,
\ee{\ga}}\frac{\li^2}{\li+\ri}$ and $y=\tg$; then we have
\begin{equation}\label{eqndoubleimp}
y \geq \frac{(3n-2)^2}{108}-\frac{(n-1)^2}{12}-\frac{n-1}{6}x
\quad \text{and } y \geq \frac{x^2}{12}.
\end{equation}
The line and the parabola, obtained by considering inequalities in
\eqnref{eqndoubleimp} as equalities, in $xy-$plane intersect at
$x=\frac{1}{3}$ and $y=\frac{1}{108}$, since $n \geq 1$. Hence,
(\ref{eqndoubleimp}) implies the result.
\end{proof}
\begin{corollary}\label{cordoubleimp1}
Let $\ga$ be a graph with $\ell(\ga)=1$. If $\ta{\ga^{DA}} \geq
\frac{1}{27}$ or $\ta{\ga^{DA,3}} \geq \frac{49}{972}$,
then $\tg \geq \frac{1}{108}$.
\end{corollary}
\begin{proof}
The result follows from
\thmref{thmdoubleimp}.
\end{proof}
%
%
In section \secref{section general DA}, we will give far-reaching generalizations of
\thmref{thmdouble} and \thmref{thmdoubledivision}.

\section{The tau constant and graph immersions}\label{section general DA}
In this section, we will define another graph operation which will be
a generalization of the process of obtaining  $\ga^{DA,n}$ from a graph $\ga$ as
presented in \secref{section DA}. Let $r(x,y)$ and $r^n(x,y)$ be the
resistance functions on $\ga$ and $\ga^{DA,n}$, respectively. First
we reinterpret the way we constructed $\ga^{DA,n}$ in order to
clarify how to generalize it.

Given a graph $\ga$ and a $n$-banana graph $\beta_n$ (the graph with
$n$ parallel edges of equal length between vertices $p$ and $q$) we
replaced each edge of $\ga$ by $\beta_{n,i}$, a copy of $\beta_n$
scaled so that each edge had length $n \li$. Then, we divided each
edge length by $n^2$ to have $\ell(\ga^{DA,n})=\ell(\ga)$. In this
operation the following features were important in enabling us to
compute $\ta{\ga^{DA,n}}$ in terms of $\tg$:
\begin{itemize}
\item We started with a graph $\ga$ and a graph $\beta_n$ with distinguished points $p$ and $q$.
\item We replaced each edge $e_i$ of $\ga$ by $\beta_{n,i}$, a copy of
$\beta_n$, scaled so that $r_{\beta_{n,i}}(p,q)=\li$.

\item After all the edge replacements were done we obtained a graph which had total
length $n^2\ell(\ga)$. We divided each edge length of this graph by
$n^2$ to obtain $\ga^{DA,n}$, so that $\ell(\ga^{DA,n})=\ell(\ga)$.
\item We kept the vertex set of $\ga$ in the vertex set of $\ga^{DA,n}$,
 $\vv{\ga}=\vv{\ga^{DA,n}}$ and for any $p$, $q$ in $\vv{\ga}$, we had
$r^n(p,q)=\frac{r(p,q)}{n^2}.$
\end{itemize}
\vskip 0.1 in Now consider the following more general setup.

Let $\ga$ and $\beta$ be two given graphs with
$\ell(\ga)=\ell(\beta)=1$. Let $p$ and $q$ be any two distinct
points in $\beta$. For every edge $e_i \in \ee{\ga}$, if $e_i$ has
length $\li$, let $\beta_i$ be the graph obtained from $\beta$ by
multiplying each edge length in $\beta$ by
$\frac{\li}{r_{\beta}(p,q)}$ where $r_{\beta}(x,y)$ is the
resistance function in $\beta$. Then
$\ell(\beta_i)=\frac{\li}{r_{\beta}(p,q)}$,  and if
$r_{\beta_i}(x,y)$ is the resistance function in $\beta_i$, then
$r_{\beta_i}(p,q)=\li$. For each edge $e_i \in \ee{\ga}$, if $e_i$
has end points $\pp$ and $\qq$, we replace $e_i$ by $\beta_i$,
identify $\pp$ with $p$ and $\qq$ with $q$. (The choice of the
labeling of the end points of $e_i$ does not change the
$\tau$-constant of the graph obtained, as the computations below
will show clearly. However, we will assume that a labeling of the end
points is given, so that the graph obtained at the end of edge
replacements will be uniquely determined.) This gives a new graph
which we will denote $\ga \star \beta_{p,q}$, and call ``the full
immersion of $\beta$ into $\ga$ with respect to $p$ and $q$'' (see
Figure \ref{fig double3}). Note that
\begin{equation}\label{eqnmag0}
\ell(\ga \star \beta_{p,q})=\sum_{e_i \in
\ee{\ga}}\ell(\beta_i)=\sum_{e_i \in
\ee{\ga}}\frac{\li}{r_{\beta}(p,q)}=\frac{\ell(\ga)}{r_{\beta}(p,q)}=\frac{1}{r_{\beta}(p,q)}.
\end{equation}
Having constructed $\ga \star \beta_{p,q}$, we divide each edge
length by $\ell(\ga \star \beta_{p,q})$, obtaining the normalized
graph $(\ga \star \beta_{p,q})^N$, with  $\ell((\ga \star
\beta_{p,q})^N)=1=\ell(\ga).$

Our goal in this section is to compute $\ta{(\ga \star \beta_{p,q})^N}$. We begin with some preliminary computations
which will also be useful in later sections.
\begin{notation}
Define $A_{p,q,\ga}:=\int_{\ga}\jj{x}{p}{q}(\ddx \jj{p}{x}{q})^2dx .$
\end{notation}
Note that $A_{p,q,\ga} \geq 0$ for any p, q $\in \ga$.
The importance of $A_{p,q,\ga}$ will be clear when we examine
its relation to $\tg$ in later sections.
In some sense it is ``the'' basic hard-to-evaluate graph integral,
and many other integrals can be evaluated in terms of it.
\begin{remark}[Scaling Property for $A_{p,q,\ga}$]\label{rem scaling}
Let $\ga$ be a graph and let $\beta$ be a graph obtained by
multiplying length of each edge in $\ee{\ga}$ by a constant $c$.
Then $\ell(\beta)=c\ell(\ga)$, $\vv{\beta}=\vv{\ga}$, $j_{x}^{\beta}(p,q) =c \jj{x}{p}{q}$, and
$A_{p,q,\beta}=c^2A_{p,q,\ga}$ for any $p$ and $q$ in $\vv{\ga}$.
\end{remark}
\begin{remark}{\label{remjpq}}
For any p, q and x $\in \ga$, $\frac{d}{dx}\jj{p}{x}{q} =
-\frac{d}{dx}\jj{q}{x}{p}$, since $r(p,q)=
\jj{p}{x}{q}+\jj{q}{x}{p}$.
\end{remark}
\begin{theorem}\label{thmremain}
For any p, q $\in \ga$, the following quantities are all equal to each other:
\begin{align*}
 & (i) \, \, A_{p,q,\ga}  \quad
 & (ii) \, \, \frac{1}{2}\int_{\ga}\jj{x}{p}{q}\DD{x}(\jj{p}{x}{q}\jj{q}{x}{p}) \quad \quad \quad
\\ & (iii) \, \, \frac{1}{2}\int_{\ga}\jj{p}{x}{q}\jj{q}{x}{p}\DD{x}\jj{x}{p}{q}
\quad & (iv) \, \, -\int_{\ga}\jj{p}{x}{q}\ddx\jj{p}{x}{q}\ddx\jj{x}{p}{q}dx \quad
\\ & (v) \, \, \int_{\ga}\jj{q}{x}{p}\ddx\jj{p}{x}{q}\ddx\jj{x}{p}{q}dx
\quad &
(vi) \, \, -\frac{r(p,q)^2}{2} + \int_{\ga}r(p,x)(\ddx \jj{p}{x}{q})^2dx
\end{align*}
\end{theorem}
\begin{proof}
(i) and (ii) are equal:
\begin{equation*}
\begin{split}
\int_{\ga} \jj{x}{p}{q} & \DD{x}(\jj{p}{x}{q}\jj{q}{x}{p}) =
\int_{\ga}\jj{x}{p}{q} \Big( \jj{q}{x}{p}\DD{x}\jj{p}{x}{q}+\jj{p}{x}{q}\DD{x}\jj{q}{x}{p}
\\ & \quad-2\ddx\jj{p}{x}{q}\ddx\jj{q}{x}{p}dx \Big), \quad \text{by \thmref{thmdelta}};
\\ &=\int_{\ga}\jj{x}{p}{q}\jj{q}{x}{p}(\dd{q}(x)-\dd{p}(x)) + \int_{\ga}\jj{x}{p}{q}\jj{p}{x}{q}(\dd{p}(x)-\dd{q}(x))
\\ & \quad -2\int_{\ga}\jj{x}{p}{q}\ddx\jj{p}{x}{q}\ddx\jj{q}{x}{p}dx, \quad \text{by \propref{prop cordjpq}};
\\ &=\jj{q}{p}{q}\jj{q}{q}{p}-\jj{p}{p}{q}\jj{q}{p}{p}+\jj{p}{p}{q}\jj{p}{p}{q}-\jj{q}{p}{q}\jj{p}{q}{q}
\\ & \quad +2\int_{\ga}\jj{x}{p}{q}(\ddx\jj{p}{x}{q})^2dx, \quad \text{by \remref{remjpq}};
\\ &=2\int_{\ga}\jj{x}{p}{q}(\ddx\jj{p}{x}{q})^2dx, \quad \text{since $\jj{q}{p}{q}=0=\jj{p}{p}{q}$};
\\ &=2A_{p,q,\ga}.
\end{split}
\end{equation*}
(ii) and (iii) are equal:
This follows from the self-adjointness of $\DD{x}$, see \propref{prop Greens identity}.
\\(iii) and (iv) are equal:
\begin{equation*}
\begin{split}
\frac{1}{2}&\int_{\ga}\jj{p}{x}{q}\jj{q}{x}{p}\DD{x}\jj{x}{p}{q} =
\frac{1}{2}\int_{\ga}\ddx\big[\jj{p}{x}{q}\jj{q}{x}{p}\big]\ddx\jj{x}{p}{q}dx
\\ & = \frac{1}{2}\int_{\ga}\ddx\jj{x}{p}{q}\big[\jj{q}{x}{p}\ddx\jj{p}{x}{q}-\jj{p}{x}{q}\ddx\jj{p}{x}{q}\big]dx,
\quad \text{by \remref{remjpq}};
\\ & = \frac{1}{2}\int_{\ga}\ddx\jj{x}{p}{q}\ddx\jj{p}{x}{q}\big[r(p,q)-2\jj{p}{x}{q}\big]dx
\\ & = \frac{r(p,q)}{2}\int_{\ga}\ddx\jj{x}{p}{q}\ddx\jj{p}{x}{q}dx-\int_{\ga}\jj{p}{x}{q}\ddx\jj{x}{p}{q}\ddx\jj{p}{x}{q}dx
\\ & = -\int_{\ga}\jj{p}{x}{q}\ddx\jj{p}{x}{q}\ddx\jj{x}{p}{q}dx, \quad \text{by \lemref{lemorthogonality}}.
\end{split}
\end{equation*}
(iv) and (v) are equal:
\begin{equation*}
\begin{split}
-\int_{\ga}&\jj{p}{x}{q}\ddx\jj{p}{x}{q}\ddx\jj{x}{p}{q}dx =
-\int_{\ga}\big[r(p,q)-\jj{q}{x}{p}\big]\ddx\jj{p}{x}{q}\ddx\jj{x}{p}{q}dx
\\ & = -r(p,q) \cdot 0 + \int_{\ga}\jj{q}{x}{p}\ddx\jj{p}{x}{q}\ddx\jj{x}{p}{q}dx, \quad \text{by \lemref{lemorthogonality}}.
\end{split}
\end{equation*}
(i) and (vi) are equal: By  \eqnref{eqn1.1},
\begin{equation*}
\begin{split}
A_{p,q,\ga} = \int_{\ga}\jj{x}{p}{q} (\ddx\jj{p}{x}{q})^2dx
= \int_{\ga}(r(p,x)-\jj{p}{x}{q}) (\ddx\jj{p}{x}{q})^2dx.
\end{split}
\end{equation*}
Hence the result follows from \corref{corjrpq}.
\end{proof}
\begin{example}\label{exdiamond}
Let $\ga$ be the graph, which
we will call the ``diamond graph", shown in
Figure \ref{fig Adiamond}. Assume the edges $\{e_1, \, \dots, e_5\}$ and the vertices $\{a, \, b, \, p, \, q\}$ are
labeled as shown. Let each edge length be $L$.
By the symmetry of the graph, edges $e_1$, $e_2$, $e_3$ and $e_4$
make the same contribution to $A_{p,q,\ga}$.
After circuit reductions and computations in Maple, we obtain that $j_{p}(x,q)$ is
constant on $e_5$, where $j_{x}(y,z)$ is the voltage function in
$\ga$. (Alternatively, $j_{p}(a,q)=j_{p}(b,q)$ by the symmetry again, so $j_{p}(x,q)$ must be constant on $e_5$.)
 Therefore, $$A_{p,q,\ga}=\int_{\ga}\jj{x}{p}{q}(\ddx
\jj{p}{x}{q})^2dx=4\int_{e_1}\jj{x}{p}{q}(\ddx \jj{p}{x}{q})^2dx .$$
Using circuit reductions and computations in Maple, one finds
$\ddx \jj{p}{x}{q}=\frac{1}{2}$ and
$\jj{x}{p}{q}=\frac{x(4L-3x)}{8L}$. Evaluating the integral gives
$A_{p,q,\ga}=\frac{L^2}{8}$.

\begin{figure}
\centering
\includegraphics[scale=0.7]{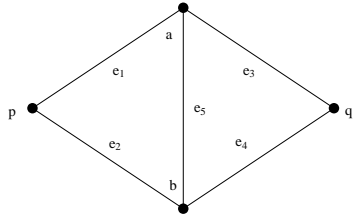} \caption{Diamond graph.} \label{fig Adiamond}
\end{figure}
\end{example}
\begin{proposition}\label{propAtree}
Let $\ga$ be a tree. Then, for any points $p$ and $q$ in
$\ga$, $A_{p,q,\ga}=0$.
\end{proposition}
\begin{proof}
Let $j_{x}(y,z)$ be the voltage function in $\ga$. Let $e_i
\in \ee{\ga}$. If $e_i$ is not between $p$ and $q$, then $\ddx j_{p}(x,q)=0$ for all $x \in e_i$.
If $e_i$ is between $p$ and $q$, then $j_{x}(p,q)=0$ for all $x \in e_i$.
Therefore, $j_{x}(p,q)(\ddx j_{p}(x,q))^2=0$ for every $x \in \ga$. This gives, by
definition, $A_{p,q,\ga}=0.$
\end{proof}
The following proposition is similar to the additive property of $\tau$.
\begin{proposition}[Additive Property for
$A_{p,q,\ga}$]\label{propAadditive} Let $\ga$, $\ga_1$ and $\ga_2$
be graphs such that $\ga=\ga_1 \cup \ga_2$ and $\ga_1 \cap
\ga_2=\{y\}$ for some $y \in \ga$. For any $p \in \ga_1$ and $q \in
\ga_2$,
$$A_{p,q,\ga}=A_{p,y,\ga_1}+A_{y,q,\ga_2}.$$
\end{proposition}
\begin{proof}
Let $j_x(p,q)$, $j_{x}^1(p,q)$ and $j_{x}^2(p,q)$ be the voltage
functions in $\ga$, $\ga_1$ and $\ga_2$ respectively.
For any $x \in \ga_1$, after circuit reduction, we obtain the first
graph in Figure \ref{fig AadditiveN}. Note that $s$ is independent of $x$, so $\ddx (s)=0.$ Also,
$j_x(p,q)=j_{x}^1(p,y)$.

Similarly, after circuit reduction, for any $x \in \ga_2$ we obtain
the second graph in Figure \ref{fig AadditiveN}. Note that $S$ is independent of $x$, so
$\ddx S=0.$ Also, $j_x(p,q)=j_{x}^2(p,y)$.
\begin{figure}
\centering
\includegraphics[scale=0.8]{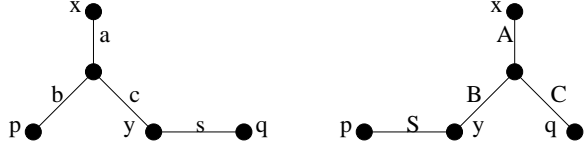} \caption{Circuit reductions for $\ga=\ga_1 \cup \ga_2$.} \label{fig AadditiveN}
\end{figure}
%
Thus
\begin{equation*}
\begin{split}
A_{p,q,\ga} &=\int_{\ga}\jj{x}{p}{q}(\ddx \jj{p}{x}{q})^2dx
\\ &=\int_{\ga_1}j_{x}^1(p,q)(\ddx j_{p}^1(x,q))^2dx+\int_{\ga_2}j_{x}^2(p,q)(\ddx j^2_{p}(x,q))^2dx.
\end{split}
\end{equation*}
Then the result follows from the definitions of $A_{p,y,\ga_1}$ and $A_{y,q,\ga_2}$.
\end{proof}
The following theorem gives value of $\ta{(\ga \star
\beta_{p,q})^N}$ in terms of $\tg$, $\ta{\beta}$, $r_{\beta}(p,q)$
and two other constants related to $\ga$ and $\beta$.
\begin{theorem}\label{thmmagnificent}
Let $\ga$ and $\beta$ be two graphs with $\ell(\ga)=\ell(\beta)=1$,
and let $p$ and $q$ be two distinct points in $\vv{\beta}$. Let
$r_{\beta}(x,y)$ be the resistance function on $\beta$. Then,
\[
\ta{(\ga \star
\beta_{p,q})^N}=\ta{\beta}-\frac{r_{\beta}(p,q)}{4}+r_{\beta}(p,q)\tg+
\frac{A_{p,q,\beta}}{r_{\beta}(p,q)}\sum_{e_i \in
\ee{\ga}}\frac{\li^2}{\li+\ri}.
\]
\end{theorem}
\begin{proof}
We will first compute $\ta{\ga \star \beta_{p,q}}$. Let $y$ be a
fixed point {\em in the vertex set $\vv{\ga}$} and let $r(x,y)$ be the
resistance function in $\ga \star \beta_{p,q}$. Then, by
\corref{lemtauformula},
\begin{equation}\label{eqnmag1}
\begin{split}
\ta{\ga \star \beta_{p,q}}&=\frac{1}{4}\int_{\ga \star
\beta_{p,q}}\big(\ddx r(x,y)\big)^2dx =\frac{1}{4}\sum_{e_i \in
\ee{\ga}} \int_{\beta_i}\big(\ddx r(x,y)\big)^2dx .
\end{split}
\end{equation}
\begin{figure}
\centering
\includegraphics[scale=0.6]{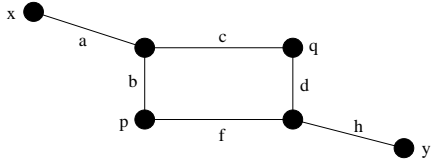} \caption{Circuit reduction
for $\ga \star \beta_{p,q}$ with reference to $p$, $q$, $y$, and
$x$.} \label{fig gendoublea}
\end{figure}
Consider a point $x \in \beta_i$.
By carrying out circuit reductions
in $\beta_i$ and in $\ga-e_i$, we obtain a network with equivalent
resistance between the points $x$, $p$, $q$, $y$ as shown in Figure
\ref{fig gendoublea}. Note that in this new circuit, the existence of
the part with edges $d$, $f$ and $h$ depends on the fact that $y$, being a point in
$\vv{\ga}$, belongs to $\ga-e_i$. It is possible that $y=p$ or
$y=q$, in which cases some of the edge lengths in $\{d, \, f, \, h \}$ are $0$.
It is also possible that $\ga-e_i$ is disconnected, in which case $d$ or $f$ will be $\infty$.
Let $j^{\beta_i}_x(y,z)$ be the voltage
function in $\beta_i$ and $R_{a_i,y}$, $R_{b_i,y}$, $R_{c_i,y}$ be
the voltages in $\ga-e_i$, using the same notation as in
\propref{proptau}. Then the resistances in Figure \ref{fig
gendoublea} are as follows: $a=\jxpq$, $b=\jpxq$, $c=\jqxp$,
$f=R_{a_i,y}$, $d=R_{b_i,y}$, $h=R_{c_i,y}$.
Note that the values in the figure are results of our conditions on
$\beta_i$ and the replacements that are made. Note also that
$b+c=r_{\beta_i}(p,q)=\li$ and $f+d=\ri$, so as $x$ varies along an
edge of $\beta_i$, we have $\ddx b=-\ddx c$.
Since $r_{\beta_j}(p,q)=L_j$ for each $e_j \in \ee{\ga-e_j}$,
$\ga \star \beta_{p,q}$ can be transformed to the circuit in \ref{fig gendoublea}.

By applying parallel reduction,
\begin{equation*}
\begin{split}
r(x,y)=a+\frac{(b+f)(c+d)}{b+c+d+f}+h
=\jxpq+\frac{(\jpxq+R_{a_i,y})(\jqxp+R_{b_i,y})}{\li+\ri}+R_{c_i,p}
.
\end{split}
\end{equation*}
Therefore,
\begin{equation*}
\begin{split}
\ddx r(x,y) = \ddx \jxpq + \frac{\jqxp+R_{b_i,y}}{\li+\ri} \ddx\jpxq
+ \frac{\jpxq+R_{a_i,y}}{\li+\ri}\ddx \jqxp.
\end{split}
\end{equation*}
Since $ \, \ddx \jqxp =-\ddx \jpxq$ and $\jpxq + \jqxp = \li$,
\begin{equation*}
\begin{split}
\ddx r(x,y)& = \ddx \jxpq +
\frac{\li-2\jpxq+R_{b_i,y}-R_{a_i,y}}{\li+\ri}\ddx \jpxq.
\end{split}
\end{equation*}
Thus,
\begin{equation}\label{eqnmag01}
\begin{split}
&\int_{\beta_i}\big(\ddx r(x,y)\big)^2dx  = \int_{\beta_i}\big(\ddx
\jxpq\big)^2dx
+ \Big[\frac{\li+R_{b_i,y}-R_{a_i,y}}{\li+\ri} \Big]^2
\int_{\beta_i}\big(\ddx \jpxq\big)^2dx
\\ & \qquad +\frac{4}{(\li+\ri)^2}\int_{\beta_i}\big[\jpxq \ddx \jpxq\big]^2dx
\\ & \qquad
+2\Big[\frac{\li+R_{b_i,y}-R_{a_i,y}}{\li+\ri} \Big] \int_{\beta_i}
\ddx \jxpq \ddx \jpxq dx
\\ & \qquad -\frac{4}{\li+\ri}\int_{\beta_i} \jpxq \ddx \jpxq \ddx \jxpq dx
\\ & \qquad
-4\frac{\li+R_{b_i,y}-R_{a_i,y}}{(\li+\ri)^2} \int_{\beta_i}\jpxq
\big[\ddx \jpxq \big]^2dx .
\end{split}
\end{equation}
On the other hand, we have
\begin{equation}\label{eqnmag02}
\begin{split}
& \text{By \corref{corjrpq}, } \quad \int_{\beta_i}\big(\ddx
\jpxq\big)^2dx=r_{\beta_i}(p,q).
\\ & \text{By \corref{corjrpq},   }\quad
\int_{\beta_i}\big[\jpxq \ddx \jpxq\big]^2dx
=\frac{1}{3}(r_{\beta_i}(p,q))^3.
\\ & \text{By \lemref{lemorthogonality},   }\quad
\int_{\beta_i} \ddx \jxpq \ddx \jpxq dx=0.
\\ & \text{By \thmref{thmremain},   }\quad \int_{\beta_i} \jpxq \ddx \jxpq \ddx \jpxq
dx=-A_{p,q,\beta_i}.
\\ & \text{By \corref{corjrpq},   } \quad \int_{\beta_i}\jpxq \big[\ddx \jpxq \big]^2dx
=\frac{1}{2}(r_{\beta_i}(p,q))^2.
\end{split}
\end{equation}
Substituting the results in \eqnref{eqnmag02} into \eqnref{eqnmag01},
and recalling $r_{\beta_i}(p,q)=\li$, gives
\begin{equation}\label{eqnmag20}
\begin{split}
\int_{\beta_i}\big(\ddx r(x,y)\big)^2dx &=\int_{\beta_i}\big(\ddx
\jxpq\big)^2dx + \Big[\frac{\li+R_{b_i,y}-R_{a_i,y}}{\li+\ri}
\Big]^2 \li
\\ &\qquad +\frac{4 \li^3}{3(\li+\ri)^2} +\frac{4 A_{p,q,\beta_i}}{\li+\ri}
-4\frac{\li+R_{b_i,y}-R_{a_i,y}}{(\li+\ri)^2} \frac{\li^2}{2}
\\ &= \int_{\beta_i}\big(\ddx \jxpq\big)^2dx + \frac{\li^3+3 \li(R_{b_i,y}-R_{a_i,y})^2}{3(\li+\ri)^2}
+\frac{4 A_{p,q,\beta_i}}{\li+\ri}.
\end{split}
\end{equation}
By applying \thmref{thmbasic} to $\beta_i$, we obtain
$$\int_{\beta_i}\big(\ddx \jxpq
\big)^2dx=4\ta{\beta_i}-r_{\beta_i}(p,q)=4\ta{\beta_i}-\li.$$
Substituting this into \eqnref{eqnmag20} and summing up over all
edges in $\ee{\ga}$ gives
\begin{equation}\label{eqnmag2}
\begin{split}
\sum_{e_i \in \ee{\ga}} \int_{\beta_i}\big(\ddx r(x,y)\big)^2dx &=
4\sum_{e_i \in \ee{\ga}}\ta{\beta_i}-\sum_{e_i \in \ee{\ga}}\li +
4\sum_{e_i \in \ee{\ga}}\frac{A_{p,q,\beta_i}}{\li+\ri}
\\ &\qquad +\frac{1}{3}\sum_{e_i \in \ee{\ga}}\frac{\li^3+3\li(R_{b_i,y}-R_{a_i,y})^2}{(\li+\ri)^2}
\\ &=4\sum_{e_i \in \ee{\ga}}\ta{\beta_i} \, -1+ 4\sum_{e_i \in \ee{\ga}}\frac{A_{p,q,\beta_i}}{\li+\ri}+ 4\tg.
\end{split}
\end{equation}
The second equality in \eqnref{eqnmag2} follows from \propref{proptau}. By using Remarks (\ref{rem tau scale-idependence}) and (\ref{rem scaling}) and the fact that
$\ell(\beta_i)=\frac{\li}{r_{\beta}(p,q)}$, we obtain
\begin{equation}\label{eqnmag3}
\begin{split}
&\ta{\beta_i}=\frac{\li}{r_{\beta}(p,q)}\ta{\beta}, \quad \text{and
} A_{p,q,\beta_i}=
\Big[\frac{\li}{r_{\beta}(p,q)}\Big]^2A_{p,q,\beta}.
\end{split}
\end{equation}
Substituting the results in \eqnref{eqnmag3} into \eqnref{eqnmag2} gives
\begin{equation}\label{eqnmag4}
\begin{split}
\sum_{e_i \in \ee{\ga}} \int_{\beta_i}\big(\ddx r(x,y)\big)^2dx =
4\frac{\ta{\beta}}{r_{\beta}(p,q)}\sum_{e_i \in \ee{\ga}}\li +
4\tg -1
+ \frac{4A_{p,q,\beta}}{(r_{\beta}(p,q))^2}\sum_{e_i \in
\ee{\ga}}\frac{\li^2}{\li+\ri}.
\end{split}
\end{equation}
Substituting \eqnref{eqnmag4} into \eqnref{eqnmag1} gives
\begin{equation}\label{eqnmag5}
\begin{split}
\ta{\ga \star \beta_{p,q}}& =
\frac{\ta{\beta}}{r_{\beta}(p,q)} + \tg -\frac{1}{4}+
\frac{A_{p,q,\beta}}{(r_{\beta}(p,q))^2}\sum_{e_i \in
\ee{\ga}}\frac{\li^2}{\li+\ri}.
\end{split}
\end{equation}
Since $\ell(\ga \star \beta_{p,q}) = \frac{1}{r_{\beta}(p,q)}$ by
\eqnref{eqnmag0}, for the normalized graph $(\ga \star \beta_{p,q})^N$ we have
$$\ta{(\ga \star \beta_{p,q})^N} = \ta{\beta} + r_{\beta}(p,q) \tg -
\frac{r_{\beta}(p,q)}{4} +
\frac{A_{p,q,\beta}}{r_{\beta}(p,q)}\sum_{e_i \in
\ee{\ga}}\frac{\li^2}{\li+\ri}.$$ This is what we want to show.
\end{proof}
\begin{theorem}\label{thm smaller tau}
Let $\ga$ be a normalized graph. Let $r(x,y)$ be the resistance
function on $\ga$, and let $p$ and $q$ be any two points in $\ga$.
Then for any $\varepsilon >0$, there exists a normalized graph
$\ga'$ such that
$$\ta{\ga'} \leq \tg - r(p,q) (\frac{1}{4}-\tg) + \varepsilon .$$
In particular, if Conjecture ~\ref{TauBound} holds with a constant
$C$, then there is no normalized graph $\beta$ with $\ta{\beta}=C$.
\end{theorem}
\begin{proof}
Let $\ga^m$ be the graph defined in \secref{section DA}. Then by
\lemref{lemdivision1},
\begin{equation}\label{eqnmag5mm}
\begin{split}
\sum_{e_k \in \ee{\ga^m}}\frac{L_k(\ga^m)^2}{L_k(\ga^m)+R_k(\ga^m)}
=\frac{1}{m}\sum_{e_i \in \ee{\ga}}\frac{\li^2}{\li+\ri}.
\end{split}
\end{equation}
Fix distinct points $p$, $q$ in $\ga$. \eqnref{eqnmag5mm} and \thmref{thmmagnificent} applied to
$\ga^m$ and $\ga$ give
\begin{equation}\label{eqnmag5mmm}
\begin{split}
\ta{(\ga^m \star \ga_{p,q})^N}=\tg-r(p,q)(\frac{1}{4}-\tg)+
\frac{A_{p,q,\ga}}{m \cdot r(p,q)}\sum_{e_i \in
\ee{\ga}}\frac{\li^2}{\li+\ri}.
\end{split}
\end{equation}
Since $\displaystyle \frac{A_{p,q,\ga}}{r(p,q)}\sum_{e_i \in
\ee{\ga}}\frac{\li^2}{\li+\ri}$ is independent of $m$, we can choose
$m$ large enough to make $\displaystyle \frac{A_{p,q,\ga}}{m \cdot
r(p,q)}\sum_{e_i \in \ee{\ga}}\frac{\li^2}{\li+\ri} \leq \varepsilon
$ for any given $\varepsilon >0$. Then taking $\ga':=(\ga^m \star \ga_{p,q})^N$ gives the inequality we wanted to show.

Suppose Conjecture ~\ref{TauBound} holds with a constant
$C$ and that $\beta$ is a normalized graph with $\ta{\beta}=C$. Then we have $\ta{\beta} \leq \frac{1}{12}$
since $\tg=\frac{1}{12}$ for the normalized circle graph $\ga$ by \corref{cor tau formula for circle}. 
Thus, $\frac{1}{4}-\ta{\beta} > 0$. Let $p$ and $q$ be distinct points in $\beta$, and let $\beta':=(\beta^m \star \beta_{p,q})^N$. For sufficiently large $m$, we have $\ta{\beta'} < \ta{\beta}$ by the inequality we proved.
This contradicts with the assumption made for $\beta$. This completes the proof of the theorem.
\end{proof}

The proof of \thmref{thmmagnificent} suggests a further
generalization of \thmref{thmmagnificent}, as follows:

Let $\ga$ be a graph with $\ell(\ga)=1$ and $e$ edges. For
each $i=1, 2, \dots, e$, suppose $\beta^i$ is a graph with
$\ell(\beta^i)=1$. Let $\pp$ and $\qq$ be any two distinct points in
$\vv{\beta^i}$, and let $r_{\beta^i}(x,y)$ be the resistance
function in $\beta^i$. By multiplying each edge length of $\beta^i$
by $\frac{\li}{r_{\beta^i}(\pp,\qq)}$ we obtain a graph which will
be denoted by $\beta_i$. Note that
$\ell(\beta_i)=\frac{\li}{r_{\beta^i}(\pp,\qq)}$ and
$r_{\beta_i}(\pp,\qq)=\li$, where $r_{\beta_i}(x,y)$ is the
resistance function in $\beta_i$. We replace each edge of $\ga$ with
$\beta_i$ and identify the end points of $e_i \in \ee{\ga}$ with the
points $\pp$ and $\qq$ in $\beta_i$ so that the resistances between
points in $\vv{\ga}$ do not change after the replacement. When edge
replacements are complete, we obtain a graph which we will denote by
$\ga \star (\beta^{1}_{p_1,q_1}\times \beta^{2}_{p_2,q_2}\times
\dots \times\beta^{e}_{p_e,q_e})$ or by $\ga \star
\prod_{i=1}^{e}\beta^{i}_{\pp,\qq}$ for short (see Figure \ref{fig replace3}).
Clearly,
\begin{equation}\label{eqnmaggen1}
\ell(\ga \star \prod_{i=1}^{e}\beta^{i}_{\pp,\qq})=\sum_{e_i \in
\ee{\ga}}\ell(\beta_i)=\sum_{e_i \in
\ee{\ga}}\frac{\li}{r_{\beta^i}(\pp,\qq)}.
\end{equation}
\begin{equation}\label{eqnmaggen2}
\begin{split}
\ta{\beta_i}=\frac{\li}{r_{\beta^i}(\pp,\qq)}\ta{\beta^i} \quad
\text{and }
A_{\pp,\qq,\beta_i}=\Big(\frac{\li}{r_{\beta^i}(\pp,\qq)}\Big)^2A_{\pp,\qq,\beta^i}.
\end{split}
\end{equation}
Let $r(x,y)$ be the resistance function in $\ga$. For any fixed $y
\in \vv{\ga}$, we can employ the same arguments as in the proof of
\thmref{thmmagnificent}. Therefore, \eqnref{eqnmag2} gives
\begin{equation}\label{eqnmaggen3}
\begin{split}
\sum_{e_i \in \ee{\ga}} \int_{\beta_i}\big(\ddx r(x,y)\big)^2dx =
-1+4\sum_{e_i \in \ee{\ga}}\ta{\beta_i}+4\sum_{e_i \in
\ee{\ga}}\frac{A_{\pp,\qq,\beta_i}}{\li+\ri} + 4\tg.
\end{split}
\end{equation}
Substituting Equations (\ref{eqnmaggen2}) into
\eqnref{eqnmaggen3} gives
\begin{equation}\label{eqnmaggen4}
\begin{split}
\sum_{e_i \in \ee{\ga}} \int_{\beta_i}\big(\ddx r(x,y)\big)^2dx &=
4\tg -1
+ \sum_{e_i \in
\ee{\ga}}\Big(\frac{4\li \ta{\beta^i}}{r_{\beta^i}(\pp,\qq)}+\frac{4 \li^2A_{\pp,\qq,\beta^i}}{(\li+\ri)(r_{\beta^i}(\pp,\qq))^2}\Big).
\end{split}
\end{equation}
Using \eqnref{eqnmag1} and \eqnref{eqnmaggen4} gives
\begin{equation}\label{eqnmaggen5}
\begin{split}
\ta{\ga \star \prod_{i=1}^{e}\beta^{i}_{\pp,\qq}} =\tg -\frac{1}{4}
+ \sum_{e_i \in
\ee{\ga}}\frac{\li}{r_{\beta^i}(\pp,\qq)}\ta{\beta^i}
+ \sum_{e_i \in
\ee{\ga}}\frac{\li^2A_{\pp,\qq,\beta^i}}{(\li+\ri)(r_{\beta^i}(\pp,\qq))^2}.
\end{split}
\end{equation}
By using \eqnref{eqnmaggen1}, we can normalize $\displaystyle \ga \star
\prod_{i=1}^{e}\beta^{i}_{\pp,\qq}$. In this way, we obtain the following theorem.
\begin{theorem}\label{thmmaggen}
Suppose $\ga$ is a normalized metrized graph with $\#(\ee{\ga})=e$. Let $\beta^i$
be a normalized metrized graph, and let $\pp$ and $\qq$ be any two points in
$\ee{\beta^i}$ for each $i=1, 2, \dots, e$. Then
\begin{equation*}
\begin{split}
\ta{\big(\ga \star & \prod_{i=1}^{e}\beta^{i}_{\pp,\qq}\big)^N} \sum_{e_i \in
\ee{\ga}}\frac{\li}{r_{\beta^i}(\pp,\qq)} =
\tg -\frac{1}{4} + \sum_{e_i \in
\ee{\ga}}\Big[\frac{\li \ta{\beta^i}}{r_{\beta^i}(\pp,\qq)}+
\frac{\li^2A_{\pp,\qq,\beta^i}}{(\li+\ri)(r_{\beta^i}(\pp,\qq))^2} \Big].
\end{split}
\end{equation*}
\end{theorem}
\begin{corollary}\label{cormaggen1}
Let $\ga$ and $\beta^1, \dots, \beta^e$ be as before. For each $\,
i\in\{1,2,\dots,e\}$, if there exist points $\pp$ and $\qq$ in
$\beta^i$ such that $r_{\beta_i}(\pp,\qq)=r$, where
$r_{\beta_i}(x,y)$ is the resistance function in $\beta^i$, then
\[
\ta{\big(\ga \star \prod_{i=1}^{e}\beta^{i}_{\pp,\qq}\big)^N}= r
\cdot \tg-\frac{r}{4}+\sum_{e_i \in
\ee{\ga}}\li\ta{\beta^i}+\frac{1}{r}\sum_{e_i \in
\ee{\ga}}\frac{\li^2}{\li+\ri}A_{\pp,\qq,\beta^i}.
\]
\end{corollary}
\begin{proof}
Setting $r_{\beta_i}(\pp,\qq)=r$ in \thmref{thmmaggen} gives the
result.
\end{proof}
\begin{corollary}\label{cormaggen2}
Let $\ga$ and $\beta$ be two normalized graphs and let
$\#(\ee{\ga})=e$. Let $r_{\beta}(x,y)$ be the resistance function
in $\beta$. For any pairs of points $\{p_1,q_1\}, \{p_2,q_2\},
\dots, \{p_e,q_e\}$ in $\beta$,
\begin{equation*}
\begin{split}
\ta{\big(\ga \star \prod_{i=1}^{e}\beta_{\pp,\qq}\big)^N}=
\ta{\beta} + \frac{1}{\sum_{e_i \in
\ee{\ga}}\frac{\li}{r_{\beta}(\pp,\qq)}} \Bigg[\tg -\frac{1}{4}
+\sum_{e_i \in
\ee{\ga}}\frac{\li^2A_{\pp,\qq,\beta}}{(\li+\ri)(r_{\beta}(\pp,\qq))^2}\Bigg].
\end{split}
\end{equation*}
\end{corollary}
\begin{proof}
Setting $\beta^i=\beta$ in \thmref{thmmaggen} gives the result.
\end{proof}
\begin{figure}
\centering
\includegraphics[scale=0.65]{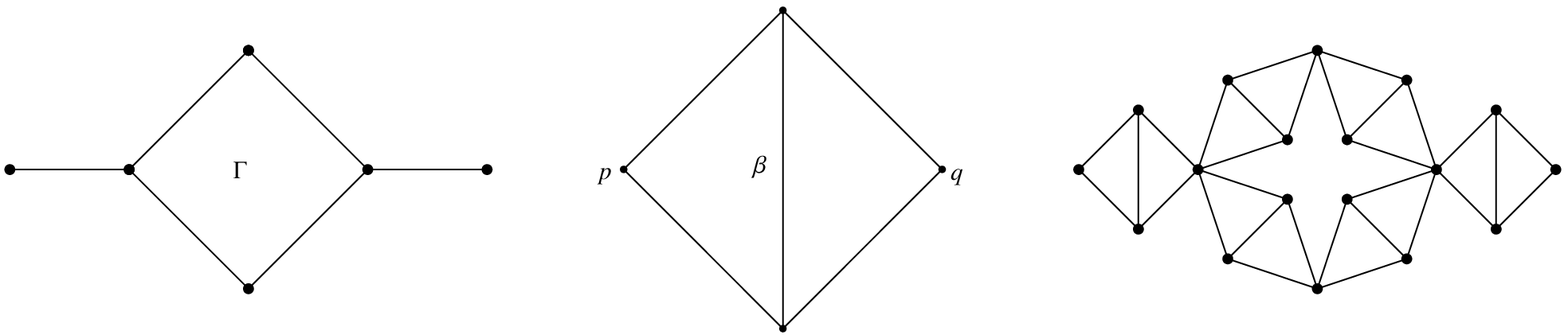} \caption{$\ga$, $\beta$ and $\ga
\star \prod_{i=1}^{e}\beta_{\pp,\qq}$.} \label{fig double3}
\end{figure}
\begin{figure}
\centering
\includegraphics[scale=0.7]{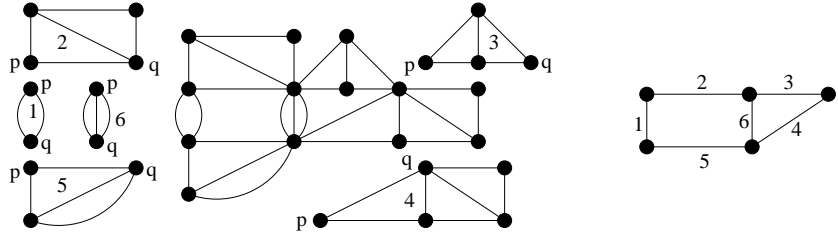} \caption{$\ga$ (edges are numbered), $\beta^i$ ($i=1, \cdots, 6$) with corresponding $p$ and $q$, and
$\ga \star \prod_{i=1}^{e}\beta^i_{\pp,\qq}$.}\label{fig replace3}
\end{figure}
%

\section{The tau constant of the union of two graphs along two
points}\label{section union along pq}
Let $\gam{1} \cup \gam{2}$ denote the union, along two points p and
q, of two connected graphs $\gam{1}$ and $\gam{2}$, so that $\gam{1}
\cap \gam{2} = \{ p,q \}$. Let $r(x,y)$, $r_{1}(x,y)$ and
$r_{2}(x,y)$ denote the resistance functions on $\gam{1} \cup
\gam{2}$, $\gam{1}$ and $\gam{2}$, respectively. Note that $\ell(\gam{1} \cup
\gam{2})=\ell(\gam{1})+\ell(\gam{2})$.
\begin{theorem}\label{thmtwopunion}
Let p, q, $r_{1}(p,q)$, $r_{2}(p,q)$, $\gam1$, $\gam2$ and
$\ta{\gam1 \cup \gam2}$ be as above. Then,
\begin{equation*}
\begin{split}
\ta{\gam1 \cup \gam2} = \ta{\gam1} + \ta{\gam2}
-\frac{r_{1}(p,q)+r_{2}(p,q)}{6} +
\frac{A_{p,q,\gam1}+A_{p,q,\gam2}}{r_{1}(p,q)+r_{2}(p,q)} .
\end{split}
\end{equation*}
\end{theorem}
\begin{proof}
Let $\ga$ be the circle graph with vertex set $\{p, q \}$, and with edge lengths $L_1=\frac{r_{1}(p,q)}{r_{1}(p,q)+r_{2}(p,q)}$
and $L_2=\frac{r_{2}(p,q)}{r_{1}(p,q)+r_{2}(p,q)}$. Let $\beta^1=\ga_1^N$ and $\beta^2=\ga_2^N$. Then the result follows by computing $\ta{ \big(\ga \star \prod_{i=1}^{2}\beta^{i}_{p,q}\big)^N}$, applying \thmref{thmmaggen}.
\end{proof}
A different proof of \thmref{thmtwopunion} can be found in \cite[page 96]{C1}.
\begin{corollary}\label{cor1twopunion}
Suppose $\ga:=\gam1=\gam2$ in \thmref{thmtwopunion}. Then, $r_{1}(x,y)=r_{2}(x,y)$ and
\begin{equation*}
\begin{split}
\ta{\ga \cup \ga} = 2\ta{\ga} -\frac{r_{1}(p,q)}{3} +
\frac{A_{p,q,\ga}}{r_{1}(p,q)}.
\end{split}
\end{equation*}
\end{corollary}
\begin{proof}
Since $\ga:=\gam1=\gam2$, clearly, we have $r_{1}(x,y)=r_{2}(x,y)$,
$\ta{\ga}= \ta{\gam1}=\ta{\gam2}$
and $A_{p,q,\ga}=A_{p,q,\gam1}=A_{p,q,\gam2}$.
\end{proof}
The following corollary of \thmref{thmtwopunion} shows how the tau constant changes by deletion of an edge when the remaining graph is connected.
\begin{corollary}\label{cor2twopunion}
Suppose that $\ga$ is a graph such that $\ga-e_{i}$ is connected, where
$e_{i} \in \ee{\ga}$ is an edge with length $\li$ and end points $p_{i}$ and
$q_{i}$. Then,
\begin{equation*}
\begin{split}
\ta{\ga} = \ta{\ga-e_{i}} +\frac{\li}{12}-\frac{\ri}{6}+
\frac{A_{p_{i},\qq,\ga-e_{i}}}{\li+\ri}.
\end{split}
\end{equation*}
\end{corollary}
\begin{proof}
Let $\gam1:=e_{i}$ and $\gam2:=\ga-e_{i}$. Therefore,
$\ta{\gam1}=\frac{\li}{4}$ by \corref{lemtauformula2}, $r_{1}(\pp,\qq)=\li$,
$r_{2}(\pp,\qq)=\ri$, and $A_{p_{i},\qq,\gam1}=0$ by \propref{propAtree}. Then by
\thmref{thmtwopunion}, we have
$\ta{\ga}=
\ta{\ga-e_{i}}+\frac{\li}{4}-\frac{1}{6}(\li+\ri)+\frac{A_{p_{i},\qq,\ga-e_{i}}}{\li+\ri}$.
This gives the result.
\end{proof}
\begin{corollary}\label{cor2twopunion2}
Suppose that $\ga$ is a graph such that $\ga-e_{i}$, for some edge
$e_{i} \in \ee{\ga}$ with length $\li$ and end points $p_{i}$ and
$q_{i}$, is connected. For the voltage function $j^{i}_{x}(y,z)$ in
$\ga-e_i$,
\begin{equation*}
\begin{split}
\ta{\ga} = \frac{1}{4}\int_{\ga-e_i}(\ddx j^{i}_{x}(\pp,\qq))^2dx
+\frac{\li+\ri}{12}+ \frac{A_{p_{i},\qq,\ga-e_{i}}}{\li+\ri} .
\end{split}
\end{equation*}
\end{corollary}
\begin{proof}
By \thmref{thmbasic},
$\ta{\ga-e_i}=\frac{1}{4}\int_{\ga-e_i}(\ddx j^{i}_{x}(\pp,\qq))^2dx +\frac{\ri}{4}$.
Substituting this into the formula of \corref{cor2twopunion}, one
obtains the result.
\end{proof}
Note that \corref{cor2twopunion2} shows that the tau constant $\tg$ approaches $\frac{\elg}{12}$ (the tau constant of a circle graph) as we increase one of the edge lengths and fix the other edge lengths.

One wonders how $\tg$ changes if one changes the length of an edge in the
graph $\ga$. \lemref{lemedgeext} below sheds some light on the answer:
\begin{lemma}\label{lemedgeext}
Let $\ga$ and $\ga'$ be two graphs such that  $\ga-e_{i}$
$\ga'-e_{i}'$ are connected, where $e_{i} \in \ee{\ga}$ is of length
$\li$ and has end points $p_{i}$, $q_{i}$ and $e_{i}' \in \ee{\ga'}$
is of length $\li+x_{i}$ and has end points $p_{i}$, $q_{i}$. Here,
$x_{i} \in \RR$ is such that $L_{i}+x_{i} \geq 0$. Suppose that
$\ga-e_{i}$ and $\ga'-e_{i}'$ are copies of each other. Then,
\begin{equation*}
\begin{split}
\ta{\ga'}=\ta{\ga}+\frac{x_{i}}{12}-\frac{x_{i}A_{\pp,\qq,\ga-e_{i}}}{(\li+\ri)(\li+\ri+x_{i})} .
\end{split}
\end{equation*}
\end{lemma}
\begin{proof}By \corref{cor2twopunion},
$\ta{\ga} = \ta{\ga-e_{i}} +\frac{\li}{12}-\frac{\ri}{6}+
\frac{A_{p_{i},\qq,\ga-e_{i}}}{\li+\ri}.$
Again, by \corref{cor2twopunion} and the fact that
$\ga-e_{i}=\ga'-e_{i}'$,
$\ta{\ga'} = \ta{\ga-e_{i}} +\frac{\li+x_{i}}{12}-\frac{\ri}{6}+
\frac{A_{p_{i},\qq,\ga-e_{i}}}{\li+x_{i}+\ri}$.
The result follows by combining these two equations.
\end{proof}
One may also want to know what happens to $\tg$ if the edge lengths are
changed successively.

Let $\ga$ be a bridgeless graph. Suppose that
$\{e_{1}, e_{2}, \dots, e_{e} \}$ is the set of edges of $\ga$ in an
arbitrarily chosen order. Recall that $e$ is the number of edges in
$\ga$. Also, $\li$ is the length of the edge $e_{i}$ with end points
$\pp$, $\qq$, for $i=1,2,\dots, e$. We define a sequence of graphs
as follows:

$\ga_{0}:=\ga$, $\ga_{1}$ is obtained from $\ga_{0}$ by changing
$L_{1}$ to $L_{1}+x_{1}$. Similarly, $\ga_{k}$ is obtained from
$\ga_{k-1}$ by changing $L_{k}$ to $L_{k}+x_{k}$ at $k-$th
step. Here, $x_{k} \in \RR$ is such that $L_{k}+x_{k} \geq 0$ for any
k. We have $\ell(\ga_{k})=\ell(\ga)+\sum_{j=1}^{k}x_{j}$. With this
change, the edge $e_{k} \in \ga_{k-1}$ becomes the edge $e_{k}' \in
\ga_{k}$, so $\ga_{k-1}-e_{k}=\ga_{k}-e_{k}'$ and
$A_{p_k,q_k,\ga_{k-1}-e_{k}}=A_{p_k,q_k,\ga_{k}-e_{k}'}$. We also
let $R'_{k}$ $(R_k)$ denote the resistance, in $\ga_{k}-e'_{k}$ (in
$\ga-e_k$), between end points of $e'_{k}$ ($e_{k}$, respectively).
Here, $k\in \{1,2,\dots,e\}$. Therefore, at the last step we obtain
$\ga_{e}$ and $\ell(\ga_{e})=\ell(\ga)+\sum_{j=1}^{e}x_{j}$.

With these notation, we have the following lemma:
\begin{lemma}\label{lemsuccessedgeext}
With the notation above,
\begin{equation*}
\ta{\ga_{e}}=\ta{\ga}+\frac{1}{12}\sum_{i=1}^{e}x_{i}
-\sum_{i=1}^{e}\frac{x_{i}A_{\pp,\qq,\ga_{i}-e_{i}'}}{(\li+\ri')(\li+\ri'+x_{i})} .
\end{equation*}
\end{lemma}
\begin{proof}
By using \lemref{lemedgeext} at each step, we obtain:
\begin{equation*}
\begin{split}
\ta{\ga_{1}}
&=\ta{\ga}+\frac{x_{1}}{12}-\frac{x_{1}A_{p_{1},q_{1},\ga_{1}-e_{1}'}}{(L_{1}+R_{1}')(L_{1}+R_{1}'+x_{1})}
\\ \ta{\ga_{2}} &=\ta{\ga_{1}}+\frac{x_{2}}{12}-\frac{x_{2}A_{p_{2},q_{2},\ga_{2}-e_{2}'}}{(L_{2}+R_{2}')(L_{2}+R_{2}'+x_{2})}
\\ &\vdots
\\ \ta{\ga_{e}} &=\ta{\ga_{e-1}}+\frac{x_{e}}{12}-\frac{x_{e}A_{p_{e},q_{e},\ga_{e}-e_{e}'}}{(L_{e}+R_{e}')(L_{e}+R_{e}'+x_{e})}
\\ \text{Then, by adding}& \text{  all of these},
\\ \ta{\ga_{e}} &=\ta{\ga}+\frac{1}{12}\sum_{i=1}^{e}x_{i}
-\sum_{i=1}^{e}\frac{x_{i}A_{\pp,\qq,\ga_{i}-e_{i}'}}{(\li+\ri')(\li+\ri'+x_{i})} .
\end{split}
\end{equation*}
\end{proof}
\begin{theorem}\label{thmbasic2}
Let $\ga$ be a bridgeless graph. Suppose that $\pp$, $\qq$ are the
end points of the edge $e_{i}$, for each $i=1,2,\dots, e$. Then,
\begin{equation*}
\tg=\frac{\ell(\ga)}{12}-\sum_{i=1}^{e}\frac{\li
A_{\pp,\qq,\ga-e_{i}}}{(\li+\ri)^2} .
\end{equation*}
\end{theorem}
\begin{proof}
Let M be a positive real number. By choosing $x_{i}=M \cdot L_{i}$
for all $i=1,2, \dots, e$ in \lemref{lemsuccessedgeext}, we obtain
$\ga_{e}$ with
$\ell(\ga_{e})=\ell(\ga)+M\sum_{j=1}^{e}L_{j}=(M+1)\ell(\ga)$. We
can also obtain $\ga_{e}$ by multiplying the length of each edge
in $\ga$ by $M+1$. Therefore, $\ta{\ga_{e}}=(1+M)\ta{\ga}$.
Then, by using \lemref{lemsuccessedgeext},
\begin{equation*}
(1+M)\tg=\tg+\frac{1}{12}M\ell(\ga)-\sum_{i=1}^{e}\frac{M \cdot L_{i}
A_{\pp,\qq,\ga_{i}-e_{i}'}}{(\li+\ri')(\li+M.\li+\ri')} .
\end{equation*}
Then,
\begin{equation*}
\tg=\frac{\ell(\ga)}{12}-\sum_{i=1}^{e}\frac{L_{i}
A_{\pp,\qq,\ga_{i}-e_{i}'}}{(\li+\ri')(\li+M.\li+\ri')} .
\end{equation*}
On the other hand, by Rayleigh's Principle (which states that if the resistances of a circuit are increased then the effective resistance between any two points can only increase, see \cite{DS} for more information), we see that $\ri \leq
\ri' \leq (1+M)\ri$.

As $M\longrightarrow 0$, we have
$\ga_{k}-e_{k}'\longrightarrow \ga-e_{k}$, $A_{\pp,\qq,\ga_{i}-e_{i}'} \longrightarrow A_{\pp,\qq,\ga-e_{i}}$,
and $\ri' \longrightarrow \ri$.
%
Hence, the result follows.
\end{proof}
\begin{corollary}\label{corbasic2}
Let $\ga$ be a bridgeless graph with total length $1$. Then, $\tg
\leq \frac{1}{12}$ .
\end{corollary}
\begin{proof}
Since $A_{\pp,\qq,\ga-e_{i}} \geq 0$ for any $i=1,2, \dots, e$,
\thmref{thmbasic2} gives the result.
\end{proof}
\begin{remark}\label{rembasic2}
The upper bound given in \corref{corbasic2} is sharp. When $\ga$ is
the circle of length $1$, $\tg=\frac{1}{12}$. For a bridgeless $\ga$,  \corref{corbasic2} improves the upper bound given in
\eqnref{FMM1}.
\end{remark}
We will give a second proof of \thmref{thmbasic2} by using
Euler's formula for homogeneous functions. A function $f: \RR^n
\rightarrow \RR$ is called homogeneous of degree $k$ if $f(\lambda
x_1,\lambda x_2, \cdots, \lambda x_n)=\lambda^k f(x_1,x_2, \cdots,
x_n)$ for $\lambda > 0$. A continuously differentiable function $f: \RR^n \rightarrow
\RR$ which is homogeneous of degree $k$ has the following property:
\begin{equation}\label{eqn Euler}
k \cdot f=\sum_{i=1}^n x_i \frac{\partial f}{\partial x_i} .
\end{equation}
\eqnref{eqn Euler} is called Euler's formula.

For a graph $\ga$ with $\# (\ee{\ga})=e$, let $\{L_1, L_2, \cdots, L_e \}$ be the edge lengths,
and let $r(x,y)$ be the resistance function on $\ga$.
For any two vertices $p$ and $q$ in $\vv{\ga}$, we have a function $R_{p,q}:\RR^e_{>0} \rightarrow
\RR$ given by $R_{pq}(L_1, L_2, \cdots, L_e)=r(p,q)$. By using circuit reductions, we can reduce
$\ga$ to a line segment with end points $p$ and $q$, and with length $r(p,q)$.
It can be seen from the edge length transformations used for circuit reductions (see \secref{section two})
that $R_{pq}(L_1, L_2, \cdots, L_e)$ is a continuously differentiable homogeneous function of degree $1$, when we consider all possible length distributions without changing the topology of the graph $\ga$.

Similarly, we have the function $T:\RR^e_{>0} \rightarrow \RR$ given by $T(L_1, L_2, \cdots, L_e)=\tg$. \propref{proptau} and the facts given in the previous paragraph imply that
$T(L_1, L_2, \cdots, L_e)$ is a continuously
differentiable homogeneous function of degree $1$, when we consider
all possible length distributions without changing the topology of
the graph $\ga$.
\begin{lemma}\label{lem diff}
Let $\ga$ be a bridgeless graph. Let $\pp$ and $\qq$ be end points
of the edge $e_i \in \ee{\ga}$, and let $L_i$ be its length for
$i=1,2, \cdots, e$. Then
$$\frac{\partial T}{\partial L_i} =\frac{1}{12}-\frac{A_{\pp,\qq,\ga-e_i}}{(\li+\ri)^2}.$$
\end{lemma}
\begin{proof}
By \corref{cor2twopunion}, $T(L_1, L_2, \cdots, L_e) =
\ta{\ga-e_{i}} +\frac{\li}{12}-\frac{\ri}{6}+
\frac{A_{p_{i},\qq,\ga-e_{i}}}{\li+\ri}, \quad$ for each $i=1,2,
\cdots, e \,$. Since $\ta{\ga-e_{i}}$, $\ri$ and
$A_{p_{i},\qq,\ga-e_{i}}$ are independent of $\li$, the result
follows.
\end{proof}
It follows from \eqnref{eqn Euler} and \lemref{lem diff} that
\thmref{thmbasic2} is nothing but Euler's formula applied to the tau
constant.

\section{How the tau constant changes by contracting edges}\label{section contraction}

For any given $\ga$, we want to understand how $\tg$ changes under
various graph operations. In the previous section, we have seen
the effects of both edge deletion on $\ga$ and changing edge lengths of $\ga$. In this section, we will
consider another operation done by contracting the lengths of
edges until their lengths become zero. First, we introduce
some notation.

Let $\oga_i$ be the graph obtained by contracting the i-th edge
$e_{i}$, $i \in \{1,2, \dots e\}$, of a given graph $\ga$ to its end
points. If $e_{i} \in \ga$ has end points $\pp$ and $\qq$, then in
$\oga_i$, these points become identical, i.e., $\pp=\qq$. Also, let
$\tga_i$ be the graph obtained from $\ga$ by identifying $\pp$ and
$\qq$, the end points of $e_{i}$. Then the edge $e_i$ of $\ga$ becomes
a self loop, which will still denoted by $e_i$, in $\tga_i$. Thus,
$\ell(\tga_i)=\ell(\oga_i)+\li = \ell(\ga)$ and $\tga_i-e_i =
\oga_i$.
\begin{lemma}\label{lemcontract1}
Let $e_i$, $\pp$, $\qq$, $\li$ and $\ri$ be as defined previously for
$\ga$. If $\ga-e_i$ is connected, then
$$
\ta{\oga_i} =\ta{\ga-e_i} - \frac{\ri}{6} +
\frac{A_{\pp,\qq,\ga-e_i}}{\ri},
\quad
\ta{\tga_i} =\ta{\ga-e_i} + \frac{\li}{12}-\frac{\ri}{6} + \frac{A_{\pp,\qq,\ga-e_i}}{\ri}.
$$
\end{lemma}
\begin{proof}
By \corref{cor2twopunion},
$\ta{\ga} = \ta{\ga-e_{i}} +\frac{\li}{12}-\frac{\ri}{6}+
\frac{A_{p_{i},\qq,\ga-e_{i}}}{\li+\ri}.$
As $\li \longrightarrow 0$, we have $\ga \longrightarrow \oga_i$, so
$\tg \longrightarrow \ta{\oga_i}$. Since $\ta{\ga-e_{i}}$, $\ri$,
$A_{p_{i},\qq,\ga-e_{i}}$ are independent of $\li$, in the limit
we obtain the following:
\begin{equation*}
\begin{split}
\ta{\oga_i} &=\ta{\ga-e_i} - \frac{\ri}{6} +
\frac{A_{\pp,\qq,\ga-e_i}}{\ri} .
\end{split}
\end{equation*}
This yields the first formula.
On the other hand, since $\tga_i-e_i$ and the self-loop $e_i$
intersect at one point, $\pp=\qq$, we can apply the additive property of
the tau constant. That is, $\ta{\tga_i}=\ta{\tga_i-e_i}+\ta{e_i}=
\ta{\oga_i}+\frac{\li}{12}$. Using this with the first formula gives
the second formula.
\end{proof}
\begin{lemma}\label{lemcontract2}
Let $e_i$, $\pp$, $\qq$, $\li$ and $\ri$ be as defined previously for
$\ga$. If $\ga-e_i$ is connected, then
$$
\ta{\ga} =\ta{\oga_i} + \frac{\li}{12} -\frac{\li
A_{\pp,\qq,\ga-e_i}}{\ri(\li+\ri)},
\qquad
\ta{\ga} =\ta{\tga_i} -\frac{\li A_{\pp,\qq,\ga-e_i}}{\ri(\li+\ri)}.
$$
\end{lemma}
\begin{proof}
By combining \corref{cor2twopunion} and \lemref{lemcontract1}, one
obtains the formulas.
\end{proof}

\section{How the tau constant changes by adding edges or identifying
points}\label{section adding edges}

Let $p$, $q$  be any two points of a graph $\ga$ and let $e^{new}$
be an edge of length $L^{new}$. By identifying end points of the
edge $e^{new}$ with p and q of $\ga$ we obtain a new graph which we
denote by $\ga_{(p,q)}$. Then,
$\ell(\ga_{(p,q)})=\ell(\ga)+L^{new}$. Also, by identifying p and q
with each other in $\ga$ we obtain a graph which we denote by
$\ga_{pq}$. Then, $\ell(\ga_{pq})=\ell(\ga)$. Note that if $p$ and $q$ are end
points of an edge $e_i \in \ga$, then $\ga_{pq}=\tga_i$, where
$\tga_i$ is as defined in \secref{section contraction}.
\begin{corollary}\label{coradding1}
Let $\ga$ be a metrized graph with resistance function
$r(x,y)$. For $p$, $q$ and $\ga_{(p,q)}$ as given above,
$$\ta{\ga_{(p,q)}}=\tg+\frac{L^{new}}{12}-\frac{r(p,q)}{6}+\frac{A_{p,q,\ga}}{L^{new}+r(p,q)}.$$
\end{corollary}
\begin{proof}
We have $\ga_{(p,q)}-e^{new}=\ga$, so the result follows from
\corref{cor2twopunion}.
\end{proof}
\begin{corollary}\label{coradding2}
Let $\ga$ be a metrized graph with resistance function
$r(x,y)$. For two distinct points $p$, $q$ and $\ga_{pq}$, we have
$$\ta{\ga_{pq}}=\tg-\frac{r(p,q)}{6}+\frac{A_{p,q,\ga}}{r(p,q)}.$$
\end{corollary}
\begin{proof}
Note that $\ga_{(p,q)} \longrightarrow \ga_{pq}$ as $L^{new}\longrightarrow 0$. Thus, we obtain what we want by using \corref{coradding1}.
\end{proof}

\section{Further properties of $A_{p,q,\ga}$}\label{section on Apq}
In this section, we establish additional properties of  $A_{p,q,\ga}$. The formulas given in this section along with the ones given previously can be used to calculate the tau constants for several
classes of metrized graphs, including graphs with vertex connectivity one or two. For metrized graphs with vertex connectivity one, we have Additivity properties for both $\tg$ and $A_{p,q,\ga}$ (see \secref{section two} and \propref{propAadditive}). For metrized graphs with vertex connectivity two, we can use the techniques developed in \secref{section general DA} and \thmref{thmtwopunion}.

First, we derive a formula for  $A_{p,q,\ga}$ for a metrized graph with vertex connectivity two.
\begin{theorem}\label{thm twopunion for Apq}
Let $\gam{1} \cup \gam{2}$ denote the union, along two points p and
q, of two connected graphs $\gam{1}$ and $\gam{2}$, so that $\gam{1}
\cap \gam{2} = \{ p,q \}$. Let $r_{1}(x,y)$ and
$r_{2}(x,y)$ denote the resistance functions on $\gam{1}$ and $\gam{2}$, respectively. Then,
\begin{equation*}
\begin{split}
A_{p,q,\gam1 \cup \gam2} = \frac{r_{2}(p,q)^2 A_{p,q,\gam1}+r_{1}(p,q)^2 A_{p,q,\gam2}}{\big(r_{1}(p,q)+r_{2}(p,q)\big)^2}
+\frac{1}{6} \Big(\frac{r_{1}(p,q)  r_{2}(p,q)}{r_{1}(p,q)+r_{2}(p,q)} \Big)^2.
\end{split}
\end{equation*}
\end{theorem}
\begin{proof}
Let $r(x,y)$ be the resistance function on $\gam{1} \cup
\gam{2}$. We have $r(p,q)=\frac{r_{1}(p,q) r_{2}(p,q)}{r_{1}(p,q)+r_{2}(p,q)}$
by parallel circuit reduction. For a metrized graph $\ga$, let $\ga_{pq}$ be the metrized graph obtained by identifying $p$ and $q$
as in \secref{section adding edges}.
By applying \corref{coradding2} to $(\gam{1} \cup \gam{2})_{pq}$,
$$\ta{(\gam{1} \cup
\gam{2})_{pq}}=\ta{\gam{1} \cup
\gam{2}}-\frac{r(p,q)}{6}+\frac{A_{p,q,\gam{1} \cup
\gam{2}}}{r(p,q)}.$$
On the other hand, $(\gam{1} \cup
\gam{2})_{pq}$ is the one point union of $(\ga_1)_{pq}$ and $(\ga_2)_{pq}$, so by the additive property of the tau constant, $\ta{(\gam{1} \cup
\gam{2})_{pq}}=\ta{(\ga_1)_{pq}}+\ta{(\ga_2)_{pq}}$. Thus by applying \corref{coradding2} to both $(\ga_1)_{pq}$ and $(\ga_2)_{pq}$,
$$\ta{(\gam{1} \cup
\gam{2})_{pq}}=\ta{\gam{1}}+\ta{\gam{2}}-\frac{r_{1}(p,q)+r_{2}(p,q)}{6} + \frac{A_{p,q,\gam{1}}}{r_{1}(p,q)}+
\frac{A_{p,q,\gam{2}}}{r_{2}(p,q)}.$$
Hence, the result follows if we compute $\ta{\gam{1} \cup
\gam{2}}$ by applying \thmref{thmtwopunion}.
\end{proof}
\begin{corollary}\label{corlem twopunion for Apq}
Let $\ga \cup \ga$ be the union of two copies of $\ga$ along any
$p$, $q$ in $\ga$. For the resistance function $r(x,y)$ in $\ga$, we
have
$$2 A_{p,q,\ga \cup \ga}= \frac{r(p,q)^2}{12}+A_{p,q,\ga}.$$
\end{corollary}
\begin{proof}
The result follows from \thmref{thm twopunion for Apq}.
\end{proof}
A different proof of \corref{corlem twopunion for Apq} can be found in \cite[page 96]{C1}.

Let $p$, $q$ be in $\ga$. Let $\cC\ga_{n}(p,q)$ be the union of n
copies of $\ga$ along $p$, $q$ in $\ga$. Note that
$\cC\ga_{2}(p,q)=\ga \cup \ga$.
\begin{theorem}\label{thm twopunion2}
Let $p$, $q$ be in $\ga$, and let $r(x,y)$ be the resistance
function in $\ga$. Let $\ga$ be a normalized graph, and let
$(\cC\ga_{2^{n}}(p,q))^N$ be the normalization of
$\cC\ga_{2^{n}}(p,q)$. Then
$$
\ta{(\cC\ga_{2^{n}}(p,q))^N}=\tg+\frac{a_{n}}{2^{n}}
\frac{A_{p,q,\ga}}{r(p,q)}+\frac{b_{n}}{2^{n}}r(p,q).
$$
where $n\geq 2$ and we have
$a_{n}=2a_{n-1} +1$,  $a_1=1$,
$b_{n}=2b_{n-1} -\frac{1}{2^{n}}+\frac{1}{6}$, and $b_1=-\frac{1}{3}$.
%
Equivalently,
$$
\ta{(\cC\ga_{2^{n}}(p,q))^N}=\tg+ \big(1-\frac{1}{2^{n}} \big)
\frac{A_{p,q,\ga}}{r(p,q)}+ \big(-\frac{1}{6}-\frac{1}{6 \cdot
2^{n}}+\frac{1}{3 \cdot 4^{n}} \big)r(p,q).
$$
\end{theorem}
\begin{proof}
Let $r_{2^k}(x,y)$ be the resistance function in
$\cC\ga_{2^{k}}(p,q)$ for $k \geq 1$ and $r_{2^0}(x,y)=r(x,y)$. Note
that $r_{2^k}(p,q)=\frac{r_{2^{k-1}}(p,q)}{2}$ for any $k\geq 1$.
Thus, applying \corref{corlem twopunion for Apq} successively gives
\begin{equation}\label{eqn twopuniona}
\frac{A_{p,q,\cC\ga_{2^{n}}(p,q)}}{r_{2^n}(p,q)}=\frac{A_{p,q,\ga}}{r(p,q)}+\frac{1}{6
} (1-\frac{1}{2^n})r(p,q).
\end{equation}
Then the result follows from \eqnref{eqn twopuniona},
\corref{cor1twopunion}, the fact that
$\ell(\cC\ga_{2^{n}}(p,q))=2^{n} \ell(\ga) = 2^{n}$, and using calculus.
\end{proof}
\begin{corollary}\label{cor twopunion}
Let $\ga$ be a normalized graph, and let $p$, $q$ be in $\ga$. Then
$$\ta{(\cC\ga_{4}(p,q))^N}=\tg +\frac{3}{4} \frac{A_{p,q,\ga}}{r(p,q)}-\frac{3}{16} r(p,q).
$$
\end{corollary}
\begin{proof}
Applying \thmref{thm twopunion2} with $n=2$ gives the result.
\end{proof}
\begin{corollary}\label{corpropAcircle}
Let $\ga$ be a circle graph. Fix $p$ and $q$ in $\ga$. Let the
edges connecting $p$ and $q$ have lengths $a$ and $b$, so
$\ell(\ga)=a+b$. Then
$A_{p,q,\ga}=\frac{a^2b^2}{6(a+b)^2}$.
\end{corollary}
\begin{proof}
Let $\ga_1$ and $\ga_2$ be two line segments of lengths $a$ and $b$.
For end points $p$ and $q$ both in $\ga_1$ and $\ga_2$,
$A_{p,q,\ga_1}=A_{p,q,\ga_2}=0$ by \propref{propAtree}. Since the circle graph $\ga$
is obtained by identifying end points of $\ga_1$ and $\ga_2$, the result follows from
\thmref{thm twopunion for Apq}.
\end{proof}
As the following lemma shows, whenever the vertices $p$ and $q$ are connected by an edge $e_{i}$ of
$\ga$, we can determine the value of $A_{p,q,\ga}$ in terms of
$A_{p,q,\ga-e_i}$ and resistance, in $\ga$, between $p$ and $q$.
\begin{lemma}\label{lemApq}
Let $e_{i} \in \ee{\ga}$ be an edge such that $\ga-e_{i}$ is
connected, where $\li$ is its length, $\ri$ is the resistance
between $p$ and $q$ in $\ga-e_i$ and $p$ and $q$ are its end points. For
the resistance function $r(x,y)$ of $\ga$,
\begin{equation*}
\begin{split}
A_{p,q,\ga}=\frac{\li^2
A_{p,q,\ga-e_{i}}}{(\li+\ri)^2}+\frac{r(p,q)^2}{6}.
\end{split}
\end{equation*}
\end{lemma}
\begin{proof}
Let $\ga_1$ be the line segment of length $\li$, and let $\ga_2$ be the graph $\ga-e_i$.
We have $A_{p,q,\ga_1}=0$ by \propref{propAtree}. Note that
$r(p,q)=\frac{\li \ri}{\li +\ri}$ by parallel circuit reduction. Inserting these values, the result follows
from \thmref{thm twopunion for Apq}.
\end{proof}
A different proof of \lemref{lemApq} can be found in \cite[Lemma 3.32]{C1}.

In the rest of this section, we will give some examples showing how the formulas we have obtained for $A_{p,q,\ga}$ and $\tg$ can be used to compute the tau constant of some graphs explicitly.
\begin{example}\label{exdiamond2}
Let $\ga$ be the Diamond graph with equal edge lengths $L$(see
\exref{exdiamond}). Let $e_5$ be the inner edge as labeled in \figref{fig Adiamond}, with end points $a$
and $b$. Then $\ga-e_5$ is a circle graph and $\ell(\ga-e_5)=4L$,
so that $\ta{\ga-e_5}=\frac{L}{3}$. Also,
$A_{a,b,\ga-e_5}=\frac{(2L)^2(2L)^2}{6(2L+2L)^2}=\frac{L^2}{6}$ by
\corref{corpropAcircle}. By parallel reduction $R_{e_5}=L$. Thus applying
\corref{cor2twopunion} to $\ga$ with edge $e_5$ gives
$ \tg = \ta{\ga-e_{5}} +\frac{L_{e_5}}{12}-\frac{R_{e_5}}{6}+
\frac{A_{p,q,\ga-e_{5}}}{L_{e_5}+R_{e_5}}
=\frac{L}{3}+\frac{L}{12}-\frac{L}{6}+
\frac{1}{L+L}\frac{L^2}{6}
=\frac{L}{3},$ i.e., $\tg=\frac{\ell(\ga)}{15}$.
\end{example}
Let $\ga$ be circle graph with $t$ vertices and $t$ edges of
length $a$. If we disconnect each vertex and
reconnect via adding a rhombus with its
short diagonal whose length is equal to
side lengths, $b$, we obtain a graph which
will be denoted by $\ga(a,b,t)$. We will call
it the ``Diamond Necklace graph'' of type
$(a,b,t)$. Figure \ref{fig diamondnecklace} gives
an example with $t=4$.
The graph $\ga(a,b,t)$ is
a cubic graph with $v=4t$ vertices and $e=6t$ edges.
\begin{figure}
\centering
\includegraphics[scale=0.7]{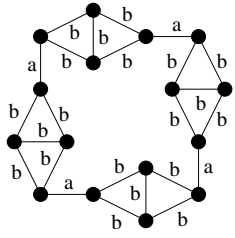} \caption{A Diamond Necklace graph, $\ga(a,b,4)$.} \label{fig diamondnecklace}
\end{figure}
\begin{example}\label{exDiamondnecklace}
Let $\ga(a,b,t)$ be a normalized Diamond Necklace graph. Let $e_a \in \ee{\ga(a,b,t)}$ be an edge of length
$a$ with end points $p$ and $q$. Note that $R_{e_a}=(t-1)a+tb$. By applying the additive property for
$A_{p,q,\ga(a,b,t)-{e_a}}$, i.e., \propref{propAadditive},  and
using \propref{propAtree}, we obtain $A_{p,q,\ga(a,b,t)-{e_a}}=t
A_{p,q,\gamma}$, where $\gamma$ is a Diamond graph with edge lengths
$b$ and $p$, $q$ as in \exref{exdiamond2}. By \exref{exdiamond},
$A_{p,q,\gamma}=\frac{b^2}{8}$. Also,
$\ta{\ga(a,b,t)-e_a}=\frac{(t-1)a}{4}+t \ta{\gamma}=\frac{(t-1)a}{4}
+t\frac{b}{3}$ by using the additive property and
\exref{exdiamond2}. Thus applying \corref{cor2twopunion} to
$\ga(a,b,t)$ with edge $e_a$ gives
\begin{equation*}
\begin{split}
\ta{\ga(a,b,t)}& = \ta{\ga(a,b,t)-e_{a}}
+\frac{L_{e_a}}{12}-\frac{R_{e_a}}{6}+
\frac{A_{p,q,\ga(a,b,t)-e_{a}}}{L_{e_a}+R_{e_a}}
\\ &=\frac{(t-1)a}{4}
+t\frac{b}{3}+\frac{a}{12}-\frac{(t-1)a+tb}{6}+
\frac{1}{a+(t-1)a+tb}t\frac{b^2}{8}
\\ &=\frac{t(a+2b)}{12}+\frac{b^2}{8(a+b)}.
\end{split}
\end{equation*}
In particular, if $\ga(a,b,t)$ is normalized, then
$1=\ell(\ga(a,b,t))=ta+5tb$ gives
$$\ta{\ga(a,b,t)}= \frac{24t^3a^2+22t^2a+4t+3-6ta+3t^2a^2}{120t(4ta+1)}.$$
\end{example}
When $\ga(a,b,t)$ is normalized, we have $b=\frac{1-a t}{5t}$ and one can show that the equality
$\frac{1}{12}\sum_{\substack{e_i \in \ga(a,b,t)}}\frac{\li^3}{(\li+\ri)^2}=\frac{4-12(a-1) t + (12a^2+24a+13)t^2 +
 a(1996 a^2-84a+91)t^3 + 8a^2 (6a+13) t^4-208a^3 t^5}{960 t^2 (4 a t+1)^2}$
holds.

In particular, when $a=\frac{1}{101}$, $b=\frac{1}{50500}$ and $t=100$ we have $\ta{\ga(a,b,t)}>\frac{1}{12.1}$ and $\displaystyle \frac{1}{12}\sum_{e_i \in \\ \ga(a,b,t)}\frac{\li^3}{(\li+\ri)^2}<\frac{1}{5000}$. Moreover, for any given $\varepsilon>0$ there are normalized diamond graphs $\ga(a,b,t)$ such that $\ta{\ga(a,b,t)}$ is close to $\frac{1}{12}$ and that $\displaystyle \frac{1}{12}\sum_{e_i \in \ga(a,b,t)}\frac{\li^3}{(\li+\ri)^2} \leq \varepsilon$. This example shows us that the method applied in the proof of \thmref{thmeqlength} can not be used to prove \conjref{TauBound} for all graphs.
\begin{proposition}\label{propAbanana}
Let $\ga$ be an m-banana graph with vertex set $\{p, q\}$ and $m$
edges. Let $r(x,y)$ be the resistance function on it. Then
$A_{p,q,\ga}=(m-1) \cdot \frac{r(p,q)^2}{6}.$
\end{proposition}
\begin{proof}
When $m=1$, $\ga$ is a line segment. In particular, it is a tree. Then the result in this case follows from
\propref{propAtree}.
When $m=2$, $\ga$ is a circle, so the result in this case follows from \corref{corpropAcircle}.
Then the general case follows by induction on $m$, if we use \lemref{lemApq}.
\end{proof}
The lower bound to the tau constant of a banana graph was studied in \cite{REU}. For a banana graph $\ga$, a \cite{REU} participant, Crystal Gordon, found by applying Lagrange multipliers that the smallest value of $\tg$ is achieved when the edge lengths are
equal to each other and the number of edges is equal to $4$ as in the following proposition. We will provide a different, shorter proof.
\begin{proposition}\label{proplembanana}
Let $\ga$ be an m-banana graph with vertex set $\{p,q\}$ and resistance function $r(x,y)$, where $m \geq 1$. Then
$\tg=\frac{\ell(\ga)}{12}-\frac{(m-2)}{6}r(p,q)$.

In particular, $\tg \geq \elg \big(\frac{1}{12}-\frac{m-2}{6 m^2}\big) \geq \frac{\elg}{16}$, 
where the first inequality holds if and only if the edge lengths of $\ga$ are all equal to each other,
and the second holds if and only if $m=4$.
\end{proposition}
\begin{proof}
By \corref{coradding2}, we have
$\ta{\ga_{pq}}=\tg-\frac{r(p,q)}{6}+\frac{A_{p,q,\ga}}{r(p,q)}$.
On the other hand, $\ta{\ga_{pq}}$ becomes one pointed union of m circles, and so by applying
additive property of the tau constant and \corref{cor tau formula for circle} we obtain $\ta{\ga_{pq}}=\frac{\ell(\ga)}{12}$.
Therefore, the equality follows from \propref{propAbanana}.

Note that the inequality was proved in \corref{cordoublebanana} when the edge lengths are equal. Let edge lengths of $\ga$
be given by $\{L_1,L_2, \cdots, L_m \}$. Then
by elementary circuit theory $r(p,q)=\frac{1}{\sum_{i=1}^m \frac{1}{L_i}}$. On the other hand, by applying the Arithmetic-Harmonic
Mean inequality we obtain $\frac{\elg}{m^2} \geq \frac{1}{\sum_{i=1}^m \frac{1}{L_i}}$, with equality if and only if the edge lengths are equal. Hence, the result follows by using the first part of the proposition and by elementary algebra.
\end{proof}


\end{document}